\documentclass[12pt]{amsart}

\usepackage{amsxtra,amssymb,amsthm,amsmath,mathabx,amscd,url,listings}
\usepackage[utf8]{inputenc}
\usepackage{eucal}
\usepackage{fullpage}
\usepackage{scrtime}
\usepackage{hyperref}

\newcounter{point}

\renewcommand{\leq}{\leqslant}
\renewcommand{\geq}{\geqslant}

\numberwithin{equation}{section}

\newcommand{\uple}[1]{\text{\boldmath${#1}$}}

\def\stacksum#1#2{{\stackrel{{\scriptstyle #1}}
{{\scriptstyle #2}}}}

\newcommand{\Cc}{\mathbf{C}}

\newcommand{\Aa}{\mathbf{A}}

\newcommand{\Zz}{\mathbf{Z}}
\newcommand{\Pp}{\mathbf{P}}
\newcommand{\Rr}{\mathbf{R}}
\newcommand{\Gg}{\mathbf{G}}

\newcommand{\Qq}{\mathbf{Q}}
\newcommand{\Fp}{{\mathbf{F}_p}}
\newcommand{\bFp}{\bar{\mathbf{F}}_p}
\newcommand{\Fpt}{{\mathbf{F}^\times_p}}

\newcommand{\bQl}{\bar{\Qq}_{\ell}}

\newcommand{\Tt}{\mathbf{T}}

\newcommand{\mods}[1]{\,(\mathrm{mod}\,{#1})}

\newcommand{\frfn}[1]{t_{{#1}}}

\DeclareMathOperator{\hypk}{Kl}
\newcommand{\HYPK}{\mathcal{K}\ell}
\DeclareMathOperator{\hypg}{Hyp}
\newcommand{\HYPG}{\sheaf{H}yp}

\newcommand{\ra}{\rightarrow}
\newcommand{\lra}{\longrightarrow}

\newcommand{\injecte}{\hookrightarrow}

\newcommand{\fleche}[1]{\stackrel{#1}{\lra}}

\DeclareMathOperator{\rank}{rank}

\DeclareMathOperator{\symk}{Sym}

\DeclareMathOperator{\frob}{\mathrm{Fr}}
\DeclareMathOperator{\Gal}{Gal}

\DeclareMathOperator{\Tr}{tr}

\DeclareMathOperator{\End}{End}
\DeclareMathOperator{\Aut}{Aut}
\DeclareMathOperator{\Autz}{Aut_0}
\DeclareMathOperator{\Autt}{Aut^{d}_0}

\DeclareMathOperator{\ft}{FT}
\DeclareMathOperator{\cond}{\mathbf{c}}

\DeclareMathOperator{\dual}{D}
\DeclareMathOperator{\std}{Std}

\newcommand{\eps}{\varepsilon}
\renewcommand{\rho}{\varrho}

\DeclareMathOperator{\cent}{Z}
\DeclareMathOperator{\SL}{SL}
\DeclareMathOperator{\GL}{GL}
\DeclareMathOperator{\PGL}{PGL}

\DeclareMathOperator{\Sp}{Sp}

\DeclareMathOperator{\SO}{SO}
\DeclareMathOperator{\Ort}{O}

\newcommand{\sheaf}[1]{\mathcal{{#1}}}
\newcommand{\tanna}[1]{\mathcal{T}({#1})}

\DeclareMathSymbol{\gena}{\mathord}{letters}{"3C}
\DeclareMathSymbol{\genb}{\mathord}{letters}{"3E}

\def\sums{\mathop{\sum \Bigl.^{*}}\limits}

\theoremstyle{plain}
\newtheorem{theorem}{Theorem}[section]
\newtheorem{lemma}[theorem]{Lemma}

\newtheorem{corollary}[theorem]{Corollary}

\newtheorem{proposition}[theorem]{Proposition}
\newtheorem*{proposition*}{Proposition}

\theoremstyle{remark}

\theoremstyle{definition}

\newtheorem{definition}[theorem]{Definition}

\newtheorem{example}[theorem]{Example}
\newtheorem{remark}[theorem]{Remark}

\renewcommand{\geq}{\geqslant}
\renewcommand{\leq}{\leqslant}

\begin{document}

\title{A study in sums of products} 
%%On certain exponential sums
 
\author{\'Etienne Fouvry}
\address{Universit\'e Paris Sud, Laboratoire de Math\'ematique\\
  Campus d'Orsay\\ 91405 Orsay Cedex\\France}
\email{etienne.fouvry@math.u-psud.fr} \author{Emmanuel Kowalski}
\address{ETH Z\"urich -- D-MATH\\
  R\"amistrasse 101\\
  CH-8092 Z\"urich\\
  Switzerland} \email{kowalski@math.ethz.ch} \author{Philippe Michel}
\address{EPFL/SB/IMB/TAN, Station 8, CH-1015 Lausanne, Switzerland }
\email{philippe.michel@epfl.ch}

\date{\today,\ \thistime} 

\thanks{Ph. M. was partially supported by
  the SNF (grant 200021-137488) and the ERC (Advanced Research Grant
  228304). \'E. F. thanks ETH Z\"urich, EPF Lausanne and the Institut
  Universitaire de France for financial support.  }

\subjclass[2010]{11T23,14F20,11G20} 

\keywords{Trace functions, \'etale cohomology, conductor, monodromy,
  $\ell$-adic sheaves, Riemann Hypothesis over finite fields,
  exponential sums}

\begin{abstract}
  We give a general version of cancellation in exponential sums that
  arise as sums of products of trace functions satisfying a suitable
  independence condition related to the Goursat-Kolchin-Ribet
  criterion, in a form that is easily applicable in analytic number
  theory.
\end{abstract}

\maketitle

{\small{
\setcounter{tocdepth}{1}
\tableofcontents}
}

\section{Introduction}\label{sec-intro}

In many (perhaps surprisingly many) applications to number theory,
exponential sums over finite fields of the type
\begin{equation}\label{eq-sp}
  \sums_{x\in\Fp}K(\gamma_1\cdot x)\cdots K(\gamma_k\cdot
  x)  e\Bigl(\frac{hx}{p}\Bigr) 
\end{equation}
arise naturally, for some positive integer $k\geq 1$, where
\begin{itemize}
\item The function $K$ is a ``trace function'' over $\Fp$, of weight
  $0$, for instance
$$
K(x)=e\Bigl(\frac{f(x)}{p}\Bigr)
$$
for some fixed polynomial $f\in\Zz[X]$, a Kloosterman sum
$$
K(x)=\frac{1}{\sqrt{p}}\sum_{y\in\Fpt}e\Bigl(\frac{y^{-1}+xy}{p}\Bigr),
$$
or its generalization to hyper-Kloosterman sums
$$
K(x)=\hypk_r(x;p)=\frac{(-1)^{r-1}}{p^{(r-1)/2}} \sum_{t_1\cdots
  t_r=x} e\Bigl(\frac{t_1+\cdots+t_r}{p}\Bigr)
$$
for some $r\geq 2$;
\item For $1\leq i\leq k$, $\gamma_i\in \PGL_2(\Fp)$ acts on $\Fp$ by
  fractional linear transformation
$$
\begin{pmatrix}a&b\\c&d
\end{pmatrix}
\cdot x=\frac{ax+b}{cx+d},
$$
for instance $\gamma_i\cdot x=a_ix+b_i$ for some $a_i\in\Fpt$ and
$b_i\in\Fp$, and the sum is restricted to those $x\in\Fp$ which are
not poles of any of the $\gamma_i$;
\item Finally, $h\in\Fp$.
\end{itemize}

The goal is usually to prove, except in special ``diagonal'' cases, an
estimate of the type
$$
\sums_{x\in\Fp}K(\gamma_1\cdot x)\cdots K(\gamma_k\cdot x)
e\Bigl(\frac{hx}{p}\Bigr) \ll\sqrt{p},
$$
where the implied constant is independent of $p$ and $h$, when $K$ has
suitably bounded ``complexity''.
\par
Note that if $K(x)$ is a Kloosterman sum, or another similar
normalized exponential sum in one variable, then opening the sums
expresses~(\ref{eq-sp}) as a $(k+1)$-variable character sum, and
(because of the normalization) the goal becomes to have square-root
cancellation with respect to all variables.
\par
We emphasize that we do not assume that the $\gamma_i$ are
distinct. Furthermore, such sums also arise with some factors
$K(\gamma_i\cdot x)$ replaced with their conjugate
$\overline{K(\gamma_i\cdot x)}$, or indeed with factors $K_i(x)$ which
are not directly related.  Such cases will be also handled in this
paper.
\par
As a sample of situations where such sums have arisen, we note:
\begin{itemize}
\item In all known proofs of the Burgess estimate for short character
  sums, one has to deal with cases where $h=0$ and
  $K_i(x)=\chi(x+a_i)$ or $\overline{\chi(x+a_i)}$ for some
  multiplicative character $\chi$ (see, e.g.,~\cite[Cor. 11.24,
  Lem. 12.8]{ant});
\item Cases where $k=2$ and $\gamma_1$, $\gamma_2$ are diagonal are
  found in the thesis of Ph. Michel and his subsequent papers,
  e.g.~\cite{michel}; 
\item For $k=2$, $\gamma_1=1$, $h=0$, we obtain the general
  ``correlation sums'' (for the Fourier transform of $K$) defined
  in~\cite{FKM1}; these are crucial to our
  works~\cite{FKM1,FKM2,FKM3};
\item Special cases of this situation of correlation sums can be found
  (sometimes implicitly) in earlier works of Iwaniec~\cite{IwActa}, of
  Pitt~\cite{pitt} and of Munshi~\cite{munshi};
\item The case $k=2$, $\gamma_1$ and $\gamma_2$ diagonal, $h$
  arbitrary and $K$ a Kloosterman sum in two variables (or a variant
  with $K$ a Kloosterman sum in one variable and $\gamma_1$,
  $\gamma_2$ not upper-triangular) occurs in the work of Friedlander
  and Iwaniec~\cite{FrIw}, and it is also used in the work of
  Zhang~\cite{zhang} on gaps between primes;
\item Cases where $k$ is arbitrary, the $\gamma_i$ are
  upper-triangular and distinct, and $h$ may be non-zero appear in the
  work of Fouvry, Michel, Rivat and S\'ark\"ozy~\cite[Lemma
  2.1]{FMRS}, indeed in a form involving different trace functions
  $K_i(\gamma_i\cdot x)$ related to symmetric powers of Kloosterman
  sums;
\item The sums for $k$ arbitrary and $h=0$, with $K$ a
  hyper-Kloosterman sum appear in the works of Fouvry, Ganguly,
  Kowalski and Michel~\cite{FGKM} and Kowalski and Ricotta~\cite{kr}
  (with $\gamma_i$ diagonal);
\item This last case, but with arbitrary $h$ and the $\gamma_i$ being
  translations also appears in the work of Irving~\cite{irving}, and
  (for very different reasons) in work of Kowalski and
  Sawin~\cite{k-sawin};.
\item Another instance, with $k=4$, $h$ arbitrary and $\gamma_i$
  upper-triangular, occurs in the work of Blomer and
  Mili\'cevi\'c~\cite[\S 11]{bm}.
\end{itemize}

The principles arising from algebraic geometry and algebraic group
theory (in particular the so-called Goursat-Kolchin-Ribet criterion,
as developed by Katz), together with the general form of the Riemann
Hypothesis over finite fields of Deligne allow for \emph{square root
  cancellation} in such sums in (also possibly surprisingly) many
circumstances. However, this principle is not fully stated in a
self-contained manner in any reference. Thus, this paper is devoted to
a review (and expansion) of these principles. We have aimed to give
statements that can be quoted easily in applications, possibly with
some additional algebraic leg-work.
\par
As already mentioned, the sums~(\ref{eq-sp}) are not the only ``sums
of products'' that appear in applications: some sums which are not of
this type are found in~\cite[Lemma 2.1]{FMRS}, in the work of Fouvry
and Iwaniec~\cite{fouvry-iwaniec-div} (estimated by Katz in the
Appendix to that paper), and in work of Bombieri and
Bourgain~\cite{bombieri-bourgain} (estimated by Katz
in~\cite{katz-bb}). In this introduction, however, we state results
only in (a slightly more general form) of~(\ref{eq-sp}), referring to
Sections~\ref{sec-general} and~\ref{sec-applications} for the general
theory and some applications, both old and new.
\par
All estimates will be derived using, ultimately, the following
application of the Riemann Hypothesis over finite fields (see
Section~\ref{sec-proofs} for a detailed explanation):

\begin{proposition}\label{pr-rh}
  Let $k\geq 1$ and let $\uple{\sheaf{F}}=(\sheaf{F}_i)$ be any
  $k$-tuple of $\ell$-adic middle-extension sheaves on $\Aa^1_{\Fp}$
  such that the $\sheaf{F}_i$ are of weight $0$, and let $\sheaf{G}$
  be an $\ell$-adic middle-extension sheaf of weight $0$. Let $K_i$ be
  the trace function of $\sheaf{F}_i$ and $M$ that of
  $\sheaf{G}$. If\footnote{\ We denote $\dual(\sheaf{G})$ the
    middle-extension dual of $\sheaf{G}$, see the notation for
    details.}
\begin{equation}\label{eq-vanish}
  H^2_c(\Aa^1\times\bFp,\bigotimes_{i}\sheaf{F}_i\otimes\dual(\sheaf{G}))=0 
\end{equation}
then we have
$$
\Bigl|\sum_{x\in\Fp}
K_1(x)\cdots K_k(x)\overline{M(x)}
\Bigr|\leq C\sqrt{p},
$$
where $C\geq 0$ depends only on $k$ and on the conductors of
$\sheaf{F}_i$ and of $\sheaf{G}$.
\end{proposition}

Thus, we will concentrate below on finding and explaining criteria
that ensure that the vanishing property~(\ref{eq-vanish}) holds,
deriving bounds for the corresponding sums from this
proposition. However, for convenience, we will state formally a number
of special cases of the resulting estimates.

% \par
% We begin by introducing some notation.  For an $\ell$-adic sheaf
% $\sheaf{F}$ over $\Fp$ with trace function $K$, for a tuple
% $\uple{\gamma}=(\gamma_1,\ldots,\gamma_k)$ of elements of
% $\PGL_2(\Fp)$, and for a tuple $\uple{\sigma}=(\sigma_1,\ldots,\sigma_k)$ of
% elements of $\Gal(\Cc/\Rr)$, we denote
% \begin{gather*}
%   \uple{\gamma}^{\oplus}\sheaf{F}= \bigoplus_i \gamma_i^*\sheaf{F}
%   \\
%   \uple{\gamma}^{\otimes}\sheaf{F}= \bigotimes_i
%   \gamma_i^*\sheaf{F}\\
%   \uple{\gamma}^{\sigma\otimes}\sheaf{F}= \bigotimes_{\sigma_i=1}
%   \gamma_i^*\sheaf{F}
% \otimes
% \bigotimes_{\sigma_i\not=1} \dual(\gamma_i^*\sheaf{F})
% \end{gather*}
% and
% $$
% (\uple{\gamma}^{*}K)(x)=K(\gamma_1\cdot x)\cdots K(\gamma_k\cdot x).
% $$
% [Better suggestions??] 
% \par
% Thus the sum~(\ref{eq-sp}) is simply (up to a factor $\sqrt{p}$) the
% value at $h\in\Fp$ of the Fourier transform of $\uple{\gamma}^*K$.

We begin by defining a class of trace function $K$ for which we can
give a general estimate for~(\ref{eq-sp}).

\begin{definition}[Bountiful sheaves]\label{def-gen-sheaf}
  We say that an $\ell$-adic sheaf $\sheaf{F}$ on $\Aa^1_{\Fp}$ is
  \emph{bountiful} provided the following conditions hold:
\begin{itemize}
\item The sheaf $\sheaf{F}$ is a middle extension, pointwise pure of
  weight $0$, of rank $r\geq 2$;
\item The geometric monodromy group of $\sheaf{F}$ is equal to either
  $\SL_r$ or $\Sp_{r}$
  % , and is equal to the arithmetic monodromy group of $\sheaf{F}$
  (we will say that $\sheaf{F}$ is of $\SL_r$-type, or
  $\Sp_r$-type, respectively);
\item The projective
  automorphism group
\begin{equation}\label{eq-autz}
\Autz(\sheaf{F})=\{\gamma\in\PGL_2(\bFp)\,\mid\,
\gamma^*\sheaf{F}\simeq \sheaf{F}\otimes \sheaf{L}\text{ for some rank
  $1$ sheaf }\sheaf{L}\}
\end{equation}
of $\sheaf{F}$ is trivial.
\end{itemize}
\end{definition}

If $\sheaf{F}$ is of $\SL_r$-type, we will also need to understand the
set
$$
\Autt(\sheaf{F})=\{\gamma\in\PGL_2(\bFp)\,\mid\,
\gamma^*\sheaf{F}\simeq \dual(\sheaf{F})\otimes \sheaf{L}\text{ for
  some rank $1$ sheaf }\sheaf{L}\},
$$
which we define for any middle-extension $\ell$-adic sheaf
$\sheaf{F}$.
\par
This definition implies that $\Autz(\sheaf{F})$ acts on
$\Autt(\sheaf{F})$ by left-multiplication: for elements
$\gamma\in\Autz(\sheaf{F})$ and $\gamma_1\in\Autt(\sheaf{F}$, we have
$\gamma_1\gamma\in\Autt(\sheaf{F})$. This action is simply transitive
(if $\gamma_1$, $\gamma_2\in \Autt(\sheaf{F})$, we get
$\gamma=\gamma_2\gamma_1^{-1}\in\Autz(\sheaf{F})$ with
$\gamma_2=\gamma\gamma_1$). This means that $\Autt(\sheaf{F})$ is
either empty or is a right coset $\xi\Autz(\sheaf{F})$ of
$\Autz(\sheaf{F})$.
\par
There is another extra property: if $\gamma\in\Autt(\sheaf{F})$, the
fact that $\dual(\dual(\sheaf{F}))\simeq\sheaf{F}$ implies that
$\gamma^2\in\Autz(\sheaf{F})$.
\par
In particular,\footnote{\ See Lemma~\ref{lm-special-coset} for a more
  general statement, based on these properties, that limits the
  possible structure of $\Autt(\sheaf{F})$.} for a sheaf with
$\Autz(\sheaf{F})=1$ (e.g., a bountiful sheaf), there are only two
possibilities: either $\Autt(\sheaf{F})$ is empty, or it contains a
single element $\xi_{\sheaf{F}}$, and the latter is an involution:
$\xi_{\sheaf{F}}^2=1$. If this second case holds, we say that
$\xi_{\sheaf{F}}$ is the \emph{special involution} of $\sheaf{F}$.
(For instance, we will see that for hyper-Kloosterman sums $\HYPK_r$
with $r$ odd, there is a special involution which is $x\mapsto -x$). 
\par
The diagonal cases, where there is no cancellation in~(\ref{eq-sp}),
will be classified by means of the following combinatorial
definitions: 

\begin{definition}[Normal tuples]\label{def-normal}
  Let $p$ be a prime, $k\geq 1$ an integer, $\uple{\gamma}$ a
  $k$-tuple of $\PGL_2(\bFp)$ and $\uple{\sigma}$ a $k$-tuple of
  $\Gal(\Cc/\Rr)=\{1,c\}$, where $c$ is complex conjugation.
\par
(1) We say that $\uple{\gamma}$ is \emph{normal} if there exists some
$\gamma\in\PGL_2(\bFp)$ such that
$$
|\{1\leq i\leq k\,\mid\, \gamma_i=\gamma\}|
$$
is odd. 
\par
(2) If $r\geq 3$ is an integer, we say that
$(\uple{\gamma},\uple{\sigma})$ is \emph{$r$-normal} if there exists
some $\gamma\in\PGL_2(\bFp)$ such that
$$
|\{1\leq i\leq k\,\mid\, \gamma_i=\gamma\}|\geq 1
$$
and
$$
|\{1\leq i\leq k\,\mid\, \gamma_i=\gamma\text{ and } \sigma_i=1 \}|-
|\{1\leq i\leq k\,\mid\, \gamma_i=\gamma\text{ and } \sigma_i\not=1
\}|\not\equiv 0\mods{r}.
$$
\par
(3) If $r\geq 3$ is an integer, and $\xi\in\PGL_2(\bFp)$ is a given
involution, we say that $(\uple{\gamma},\uple{\sigma})$ is
\emph{$r$-normal with respect to $\xi$} if there exists some
$\gamma\in\PGL_2(\bFp)$ such that
$$
|\{1\leq i\leq k\,\mid\, \gamma_i=\gamma\}|\geq 1
$$
and
\begin{equation}\label{eq-r-normal}
\Bigl(\sum_{\stacksum{1\leq i\leq k}{(\gamma_i,\sigma_i)=(\gamma,1)}}1
+
\sum_{\stacksum{1\leq i\leq k}{(\gamma_i,\sigma_i)=(\xi\gamma,c)}}1\Bigr)
-
\Bigl(\sum_{\stacksum{1\leq i\leq k}{(\gamma_i,\sigma_i)=(\gamma,c)}}1
+\sum_{\stacksum{1\leq i\leq k}{(\gamma_i,\sigma_i)=(\xi\gamma,1)}}1\Bigr)
\not\equiv 0\mods{r}.
  % |\{1\leq i\leq k\,\mid\, \gamma_i\in\{\gamma,\xi\gamma\}\text{ and }
  % \sigma_i=1 \}|- \\
  % |\{1\leq i\leq k\,\mid\, \gamma_i\in\{\gamma,\xi\gamma\}\text{ and }
  % \sigma_i\not=1 \}|\not\equiv 0\mods{r}.
\end{equation}
%% TODO: check ordering of product!
\end{definition}

\begin{example}
  (1) The basic example of a pair $(\uple{\gamma},\uple{\sigma})$
  which is not $r$-normal arises when $k$ is even and it is of the
  form
$$
((\gamma_1,\gamma_1,\ldots,\gamma_{k/2},\gamma_{k/2}),(1,c,\ldots,
1,c))
$$
since we then have
$$
|\{1\leq i\leq k\,\mid\, \gamma_i=\gamma\text{ and } \sigma_i=1 \}|
=|\{1\leq i\leq k\,\mid\, \gamma_i=\gamma\text{ and } \sigma_i=c\}|
$$
for any $\gamma\in \{\gamma_1,\ldots,\gamma_{k/2}\}$.
\par
(2) Let $\xi\in\PGL_2(\bFp)$ be an involution. Some basic examples of
pairs $(\uple{\gamma},\uple{\sigma})$ which are not $r$-normal with
respect to $\xi$ are the following:
\begin{itemize}
\item If $k$ is even, pairs
$$
((\gamma_1,\xi\gamma_1,\ldots,\gamma_{k/2},\xi\gamma_{k/2}),(1,1,\ldots,
1,1))
$$
(for instance, if the $\gamma_i$ are distinct, the left-hand side
of~(\ref{eq-r-normal}) is then
$$
(1+0)-(0+1)=0
$$
for each $\gamma\in\{\gamma_1,\ldots,\gamma_{k/2}\}$),
\item For $r=3$, $k=7$, pairs
$$
((\gamma,\xi\gamma,\xi\gamma,\gamma,\gamma,\xi\gamma,\gamma),
 (1,     c,        c,        c,     1,     1,        1))
$$
where the left-hand side of~(\ref{eq-r-normal}) for $\gamma$
(resp. $\xi\gamma$) is
$$
(3+2)-(1+1)=3\equiv 0\mods{3}\quad\quad
\text{(resp. $(1+1)-(2+3)=-3$).}
$$
\end{itemize}
\end{example}

After these definitions, we have first an abstract statement, from
which estimates follow immediately from Proposition~\ref{pr-rh}. In
this statement, for a sheaf $\sheaf{F}$ and $\sigma\in\Aut(\Cc/\Rr)$,
we denote $\sheaf{F}^{\sigma}=\sheaf{F}$ if $\sigma$ is the identity,
and $\sheaf{F}^{\sigma}=\dual(\sheaf{F})$ if $\sigma=c$ is complex
conjugation.

\begin{theorem}[Abstract sums of products]\label{th-main1}
  Let $p$ be a prime and let $\sheaf{F}$ be a bountiful $\ell$-adic
  sheaf on $\Aa^1_{\Fp}$.
\par
\emph{(1)} Assume that $\sheaf{F}$ is of $\Sp_r$-type. For every
$k\geq 1$, every $k$-tuple $\uple{\gamma}$ of elements in
$\PGL_2(\bFp)$, and every $h\in\Fp$, we have
$$
H^2_c(\Aa^1\times\bFp, \bigotimes_{1\leq i\leq k}
\gamma_i^*\sheaf{F}\otimes \sheaf{L}_{\psi(hX)})=0
$$
provided that either $\uple{\gamma}$ is normal or that $h\not=0$.
\par
\emph{(2)} Assume that $\sheaf{F}$ is of $\SL_r$-type. For every
$k\geq 1$, for all $k$-tuples $\uple{\gamma}$ of elements of
$\PGL_2(\bFp)$ and $\uple{\sigma}$ of elements of $\Aut(\Cc/\Rr)$, and
for all $h\in\Fp$, we have
$$
H^2_c(\Aa^1\times\bFp, \bigotimes_{1\leq i\leq k}
\gamma_i^*(\sheaf{F}^{\sigma})\otimes \sheaf{L}_{\psi(hX)})=0
$$
provided that either $h\not=0$, or that $h=0$ and either
\begin{itemize}
\item $\sheaf{F}$ has no special involution, and
  $(\uple{\gamma},\uple{\sigma})$ is $r$-normal;
\item $\sheaf{F}$ has a special involution $\xi$, $p>r$, and
  $(\uple{\gamma},\uple{\sigma})$ is $r$-normal with respect to $\xi$.
\end{itemize}
\end{theorem}

To be concrete, we get:

\begin{corollary}[Bountiful sums of products]\label{cor-concrete}
  Let $p$ be a prime and let $K$ be the trace function modulo $p$ of a
  bountiful sheaf $\sheaf{F}$ with conductor $c$. Then, for any $k\geq
  1$, there exists a constant $C=C(k,c)$ depending only on $c$ and $k$
  such that:
\par
\emph{(1)} If $\sheaf{F}$ is self-dual, so that $K$ is real-valued,
then for any $k$-tuple $\uple{\gamma}$ of elements of $\PGL_2(\bFp)$
and for any $h\in\Fp$, provided that \emph{either} $\uple{\gamma}$ is
normal, \emph{or} $h\not=0$, we have
$$
\Bigl| \sums_{x\in\Fp}K(\gamma_1\cdot x)\cdots K(\gamma_k\cdot x)
e\Bigl(\frac{hx}{p}\Bigr) \Bigr| \leq C\sqrt{p}.
$$
\par
\emph{(2)} If $\sheaf{F}$ is of $\SL_r$-type with $r\geq 3$, and
$p>r$, then for $k$-tuples $\uple{\gamma}$ of elements of
$\PGL_2(\bFp)$ and $\uple{\sigma}$ of $\Aut(\Cc/\Rr)$, and for any
$h\in\Fp$, provided either that $(\uple{\gamma},\uple{\sigma})$ is
$r$-normal, or $r$-normal with respect to the special involution of
$\sheaf{F}$, if it exists, or that $h\not=0$, we have
$$
\Bigl| \sums_{x\in\Fp}K(\gamma_1\cdot x)^{\sigma_1}\cdots
K(\gamma_k\cdot x)^{\sigma_k} e\Bigl(\frac{hx}{p}\Bigr) \Bigr| \leq
C\sqrt{p}.
$$
% \par
% \emph{(3)} In all cases not covered by~\emph{(1)} or~\emph{(2)}, we
% have
% $$
% \Bigl| \sums_{x\in\Fp}K(\gamma_1\cdot x)^{\sigma_1}\cdots
% K(\gamma_k\cdot x)^{\sigma_k} e\Bigl(\frac{hx}{p}\Bigr)-p \Bigr| \leq
% C\sqrt{p}.
% $$
\end{corollary}

This is intuitively best possible, because if $\sheaf{F}$ is self-dual
and $\uple{\gamma}$ is not normal, so that the distinct elements
$\gamma_j$ in $\gamma$ appear each with even multiplicity $2n_j$, we
get for $h=0$ the sum
$$
 \sums_{x\in\Fp}\prod_{j}K(\gamma_j\cdot x)^{2n_j}
$$
in which there is no cancellation to be expected. The corresponding
optimality holds for sheaves of $\SL_r$-type, but this is less
obvious.
\par
It is sometimes important to determine even in this case what is the
main term that may arise (e.g., in~\cite{FGKM,kr}, this allows one to
identify the main term in a central limit theorem).  This is given by
the following statements.

\begin{corollary}\label{cor-concrete2}
  Let $p$ be a prime and let $K$ be the trace function modulo $p$ of a
  bountiful sheaf $\sheaf{F}$ with conductor $c$. Assume furthermore:
\begin{itemize}
\item  That the arithmetic monodromy group of $\sheaf{F}$ is equal to the
  geometric monodromy group,
\item If $\sheaf{F}$ is of $\SL_r$-type and has a special involution
  $\xi$, that
$$
\xi^*\sheaf{F}\simeq \dual(\sheaf{F}).
$$
\end{itemize}
\par
Then, for any $k\geq 1$, there exists a constant $C=C(k,c)$ depending
only on $c$ and $k$ such that:
\par
\emph{(1)} If $\sheaf{F}$ is of $\Sp_{2g}$-type, then for any
$k$-tuple $\uple{\gamma}$ of elements of $\PGL_2(\bFp)$ which is not
normal and for any $h\in\Fp$, there exists an integer
$m(\uple{\gamma})\geq 1$ such that
$$
\Bigl| \sums_{x\in\Fp}K(\gamma_1\cdot x)\cdots K(\gamma_k\cdot
x)-m(\uple{\gamma}) p \Bigr| \leq C\sqrt{p}.
$$
\par
If $k$ is even and $\uple{\gamma}$ consists of pairs of $k/2$ distinct
elements, then $m(\uple{\gamma})=1$. 
In general,
$$
m(\uple{\gamma})=\prod_{\gamma\in\uple{\gamma}} A(n_{\gamma})
$$
where $\gamma$ runs over all elements occuring in the tuple
$\uple{\gamma}$, $n_{\gamma}$ is the multiplicity of $\gamma$ in the
tuple and $A(n)$ is the multiplicity of the trivial representation of
$\Sp_{2g}$ in the $n$-th tensor power of the standard representation
of $\Sp_{2g}$.
\par
\emph{(2)} If $\sheaf{F}$ is of $\SL_r$-type with $r\geq 3$, then for
$k$-tuples $\uple{\gamma}$ of elements of $\PGL_2(\bFp)$ and
$\uple{\sigma})$ of $\Aut(\Cc/\Rr)$, such that
$(\uple{\gamma},\uple{\sigma})$ is not $r$-normal, or not $r$-normal
with respect to the special involution of $\sheaf{F}$ if it exists,
there exists an integer $m(\uple{\gamma},\uple{\sigma})\geq 1$ such
that
$$
\Bigl| \sums_{x\in\Fp}K(\gamma_1\cdot x)^{\sigma_1}\cdots
K(\gamma_k\cdot x)^{\sigma_k} -m(\uple{\gamma},\uple{\sigma})p\Bigr| \leq
C\sqrt{p}.
$$
\par
If $k$ is even, $\uple{\gamma}$ consists of $k/2$ pairs of elements
which are distinct or distinct modulo the special involution if it
exists, and for each such pair $(\gamma_i,\gamma_j)$, one of
$\sigma_i$ is the identity and the other is $c$, then
$m(\uple{\gamma},\uple{\sigma})=1$. Otherwise,
$m(\uple{\gamma},\uple{\sigma})$ is bounded in terms of $k$ and $r$
only.
\end{corollary}

The proofs of Theorem~\ref{th-main1}, Corollaries~\ref{cor-concrete}
and~\ref{cor-concrete2} will be found in Section~\ref{sec-proofs},
after we develop a more general framework in
Section~\ref{sec-general}. Many examples of (trace functions of)
bountiful sheaves, and also of the more general situation of the next
section, together with more statements of the resulting estimates, are
found in Section~\ref{sec-examples}.  Readers may wish to first read
through this last section in order to see more examples of the
estimates we obtain.
\par
There is a certain inevitable tension in this paper between the fact
that, on the one hand, we deal with rather general phenomena, and on
the other hand most applications involve extremely concrete special
cases. In Section~\ref{sec-howto}, we try to explain how one can, in
practice, begin to investigate a given sum with the help of the tools
described in this paper.

\subsection*{Notation and conventions}

(1) An $\ell$-adic sheaf over an algebraic variety $X$ defined over
$\Fp$ will always mean a constructible $\bar{\Qq}_{\ell}$-sheaf for
some $\ell\not=p$; whenever the trace function of such sheaves are
mentioned, it is assumed that an isomorphism
$\iota\,:\,\bar{\Qq}_{\ell}\lra \Cc$ has been chosen once and for all,
and that the trace function is seen as complex-valued through this
isomorphism.
\par
(2) A tuple $\uple{a}=(a_1,\ldots, a_k)$ (with $a_i$ in any set $A$)
is said to be primitive if all components are distinct. The
multiplicity in $\uple{a}$ of any element $a\in A$ is the number of
$i$ such that $a_i=a$. We will sometimes write $a\in\uple{a}$ (or
$a\notin\uple{a}$) to indicate that an element $a$ is (or is not)
among these components.  A subtuple $\uple{b}$ will mean any $l$-tuple
with $l\leq k$ such that all components of $\uple{b}$ are taken among
the $a_i$, with multiplicity at most that of $a_i$ in $\uple{a}$. We
will sometimes implicitly allow the components to be rearranged, which
will not affect any argument since all components will play symmetric
roles, or explicitly denote $\uple{a}\sim \uple{a}'$ to say that
$\uple{a}$ and $\uple{a}'$ differ only up to order (this includes
equality of multiplicity). Similarly, a sum (resp. product, tensor
product) product over $a\in\uple{a}$ means a sum (resp. product,
tensor product) with multiplicity, e.g.
$$
\sum_{a\in (1,1,2)}a^2=1^2+1^2+2^2.
$$
\par
(3) For a lisse sheaf $\sheaf{F}$ (resp. a middle-extension sheaf
$\sheaf{F}$ on $\Aa^1$) we denote by $\dual(\sheaf{F})$ the dual lisse
sheaf (resp. the middle-extension dual $j_*(\dual(j^*\sheaf{F}))$
where $j\,:\, U\injecte \Aa^1$ is the open immersion of a dense open
set where $\sheaf{F}$ is lisse).  If $\rho$ is a finite-dimensional
representation of a group $G$, we denote by $\dual(\rho)$ the
contragredient representation.
\par
(4) We denote by $\cent(G)$ the center of a group $G$, and by $G^0$
the connected component of the identity in a topological or algebraic
group $G$.

\subsection*{Acknowledgements}
 
Thanks to Z. Rudnick for feedback and suggestions concerning the
paper. Thanks also to A. Irving for asking a question that led us to
find a slip in a previous version.

\section{A general framework}\label{sec-general}

We provide in this section, and the next, a very general statement
concerning sheaves with trace functions of the type appearing
in~(\ref{eq-sp}). This will be presented in a purely algebraic manner,
and later sections will provide the diophantine interpretation that
leads to the results of the first section, as well as to more general
statements, which will be explained in the later sections.
\par
We first make a definition that encapsulates some of the content of
the Goursat-Kolchin-Ribet criterion of Katz (see\cite[\S
1.8]{katz-esde}):

\begin{definition}[Generous tuple]
  Let $k\geq 1$ be an integer and $p$ a prime. Let
  $U\subset\Aa^1_{\Fp}$ be a dense open set. Let
  $\sheaf{F}=(\sheaf{F}_i)$ be a tuple of $\ell$-adic middle-extension
  sheaves on $\Aa^1_{\Fp}$, all lisse on $U$. Denote by
$$
\rho_i\,:\, \pi_1(U\times\bFp,\bar{\eta})\lra \GL(V_i)
$$
the $\ell$-adic representations corresponding to $\sheaf{F}_i$, and
$$
\rho=\bigoplus_{1\leq i\leq k}\rho_i
$$
\par
We say that $\uple{\sheaf{F}}$ is \emph{$U$-generous} if:
\begin{enumerate}
\item The sheaves $\sheaf{F}_i$ are geometrically irreducible and
  pointwise pure of weight $0$ on $U$;
\item For all $i$, the normalizer of the connected component of the
  identity $G_i^0$ of the geometric monodromy group $G_i$ of
  $\sheaf{F}_i$ is contained in $\Gg_m G_i^0\subset \GL(V_i)$ and its
  Lie algebra is simple (in particular, $G^0_i$ acts irreducibly on
  $V_i$);
\item For all $i\not=j$, the pairs $(G_i^0,\std_i)$ and
  $(G_j^0,\std_j)$ are Goursat-adapted in the sense
  of~\cite[p. 24]{katz-esde}, where $\std_i$ denotes the tautological
  representations $G_i\subset \GL(V_i)$;
\item Let $G$ be the Zariski closure of the image of $\rho$ and let
  $\tilde{\rho_i}\,:\, G\lra \GL(V_i)$ be the representation such that
  $\rho_i$ is the composition
$$
\pi_1(U\times\bFp,\bar{\eta})\fleche{\rho}G\fleche{\tilde{\rho}_i}
\GL(V_i)\ ;
$$
then for all $i\not=j$, and all $1$-dimensional characters $\chi$ of
$G$, there is no isomorphism
\begin{equation}\label{eq-gkr-cond}
  \tilde{\rho}_i\simeq \tilde{\rho}_j\otimes \chi,
  \text{ or } \dual(\tilde{\rho}_i)\simeq
  \tilde{\rho}_j\otimes\chi
\end{equation}
as representations of $G$.
\end{enumerate}
\par
We say that $\uple{\sheaf{F}}$ is \emph{strictly $U$-generous} if it
is generous and the monodromy groups $G_i$ are connected.
% In this case, it is enough in the last condition to check that there
% is no isomorphism
% $$
% \rho_i\simeq \rho_j,\text{ or } \dual(\rho_i)\simeq \rho_j,
% $$
% (since $G$ is then semisimple and connected, hence has no
% $1$-dimensional characters).
\end{definition}

\begin{remark}
The last condition holds in particular if, for $i\not=j$, there is no
rank $1$ sheaf $\sheaf{L}$ such that
$$
\sheaf{F}_i\simeq \sheaf{F}_j\otimes \sheaf{L}, \text{ or }
\dual(\sheaf{F}_i)\simeq \sheaf{F}_j\otimes\sheaf{L},
$$
and we will usually check it in this form.
\end{remark}

\begin{example}\label{ex-old}
We just give quick examples here, leaving more detailed discussions to
Section~\ref{sec-examples}.
\par
(1) Let $U=\Gg_m$. Given $n\geq 1$ even (resp. odd) and a $k$-tuple
$(a_i)$ of distinct elements of $\Fpt$ (resp. elements distinct modulo
$\pm 1$), we take $\sheaf{F}_i=[\times a_i]^*\HYPK_n$, where $\HYPK_n$
is the $n$-variable Kloosterman sheaf with trace function
$\hypk_n(x;p)$ (see Section~\ref{sec-examples}).
\par
Then $(\sheaf{F}_i)$ is strictly $U$-generous. This follows from the
theory of Kloosterman sheaves, in particular the computation of the
geometric monodromy groups by Katz~\cite{katz-gkm}, and the fact that
there does not exist a rank $1$ sheaf $\sheaf{L}$ and a geometric
isomorphism
$$
[\times a]^*\HYPK_n\simeq \HYPK_n\otimes\sheaf{L}\text{ or } [\times
a]^*\HYPK_n\simeq \dual(\HYPK_n)\otimes\sheaf{L},
$$
for $a\not=1$ if $n$ is even, and for $a\notin \{\pm 1\}$ if $n$ is
odd. (In other words, we have $\Autz(\HYPK_r)=1$, and for $r\geq 3$
odd, $\Autt(\HYPK_r)$ contains the unique special involution $x\mapsto
-x$; see Section~\ref{sec-examples} for details).
\par
(2) Given $\sheaf{F}_0$ self-dual and lisse on $\Gg_m$, with geometric
monodromy group equal to $\Sp_{r}$, such that the projective
automorphism group of $\sheaf{F}_0$ is trivial, and a $k$-tuple
$(a_i)$ of distinct elements of $\Fpt$, we may take
$\sheaf{F}_i=[\times a_i]^*\sheaf{F}_0$ on $U=\Gg_m$, and
$(\sheaf{F}_i)$ is then strictly $\Gg_m$-generous.
% Here the projective automorphism group of a sheaf $\sheaf{F}$ on
% $\Aa^1_{\bFp}$ is defined to be
% $$
% \Autz(\sheaf{F})=\{\gamma\in\PGL_2(\bFp)\,\mid\,
% \gamma^*\sheaf{F}\simeq \sheaf{F}\otimes \sheaf{L}\text{ for some rank
%   $1$ sheaf }\sheaf{L}\}.
% $$
\par
(3) Given $\sheaf{F}_0$ lisse on $\Gg_m$ with geometric monodromy
group $\Gg_0$ \emph{containing} $\SL_{r}$ for some $r\geq 3$, such
that
$$
\Autz(\sheaf{F}_0)\cap \Tt=1,
$$ 
where $\Tt\subset \PGL_2$ is the diagonal torus, 
% $$
% \Autz(\sheaf{F})=\{\gamma\in\PGL_2(\bFp)\,\mid\,
% \gamma^*\sheaf{F}\simeq \sheaf{F}\otimes \sheaf{L}\text{ for some rank
%   $1$ sheaf }\sheaf{L}\}.
% $$
and a $k$-tuple $\uple{a}=(a_i)$ of elements of $\Fpt$, 
%  distinct modulo
% $\pm 1$
then the tuple $([\times a_i]^*\sheaf{F}_0)$ is $\Gg_m$-generous.
\par
Indeed, all conditions of the definition are clearly met, except maybe
for the non-existence of isomorphisms
$$
\dual(\rho_i)\simeq \rho_j\otimes\chi
$$
for $i\not=j$. But restricting such an isomorphism to the inverse
image of $\SL_r\subset \Gg_0$, this would imply that the standard
representation of $\SL_r$ is self-dual, which is not the case (since
the restriction of $\sheaf{L}$ to this subgroup must be trivial, as it
factors through a character of $\SL_r$).
% But such an isomorphism would imply that there exists a
% rank $1$ lisse sheaf $\sheaf{L}$ on $\Gg_m$ and a geometric
% isomorphism
% $$
% [\times a_i]^*\dual(\sheaf{F}_0)\simeq [\times
% a_j]^*\sheaf{F}_0\otimes\sheaf{L}, 
% $$
% and this implies that
% $$
% [\times (a_ia_j^{-1})^2]\in \Autz(\sheaf{F}_0)=1,
% $$
% so that $a_i=\pm a_j$, contradicting our assumption on the tuple
% $\uple{a}$.
\par
(4) Given a $U$-generous tuple (resp strictly $U$-generous tuple), any
subtuple is still $U$-generous (resp. strictly
$U$-generous). Similarly, if $V\subset U$ is another dense open set,
the restrictions to $V$ of a $U$-generous tuple is $V$-generous (and
similarly for strictly generous tuples).
\end{example}

We now come back to the development of the general theory. The crucial
point is the following lemma:

\begin{lemma}[Katz]\label{lm-gkr}
  Let $\uple{\sheaf{F}}$ be $U$-generous. Then the connected component
  of the identity of the geometric monodromy group $G$ of the sheaf
$$
\bigoplus_{i}\sheaf{F}_i
$$
on $U$ is equal to the product
$$
G^0=\prod_{1\leq i\leq k}G_i^0
$$
of the connected components of the geometric monodromy groups $G_i$ of
$\sheaf{F}_i$. If $\uple{\sheaf{F}}$ is strictly generous, then
$G=G^0$. 
\par
Let $\pi\,:\, V\times\bFp\ra U\times \bFp$ be the finite abelian
\'etale covering corresponding to the surjective homomorphism
$$
\pi_1(U\times\bFp,\bar{\eta})\lra G/G^0,
$$
so that $V=U$ and $\pi$ is the identity on $U\times\bFp$ if
$\uple{\sheaf{F}}$ is strictly $U$-generous. Then the geometric
monodromy group of
$$
\pi^*\Bigl(\bigoplus_{i}\sheaf{F}_i\Bigr)
$$ 
is equal to $G^0$. Furthermore, the restriction to $G^0$ of any
  irreducible representation of $G$ is irreducible.
\end{lemma}

\begin{proof}
  In view of the definition, the computation of the monodromy groups
  is a special case of the Goursat-Kolchin-Ribet Proposition of
  Katz~\cite[Prop. 1.8.2]{katz-esde} (noting that, with the notation
  there, if the normalizer of $G_i^0$ is contained in $\Gg_mG_i^0$,
  then $G_i^0$ acts irreducibly on $V_i$, because any
  subrepresentation is stable under the action of $\Gg_mG^0_i\supset
  N_{\GL(V_i)}G_i^0\supset G_i$).
\par
For the last part, let $\tau$ be an irreducible representation of
$G$. Note that
$$
G\subset \prod_i (\Gg_m G_i^0)\subset \cent(G)G^0
$$
by the second condition in the definition of a generous tuple, and the
fact that any $g\in G$ is of the form
$$
g=(\xi_i g_i)
$$ 
for some $\xi_i\in\Gg_m\cap G_i\subset \cent(G_i)$ and $g_i\in G_i^0$,
so that $g=zh$ with $z=(\xi_i)\in \cent(G)$ and $h=(g_i)\in G^0$.  It
follows that for any $g=zh\in G$, we have
$$
\tau(g)=\tau(zh)=\tau(z)\tau(h).
$$
\par
Since $\tau(z)$ is a scalar (because $\tau$ is $G$-irreducible and $z$
is central), we see that any $G^0$-invariant subspace is also
$G$-invariant.
\end{proof}

\begin{remark}\label{rm-so}
(1)  Note that even if the $G_i$ are connected, one must check the
  condition~(\ref{eq-gkr-cond}) with characters $\chi$ (although each
  $G_i$, being semisimple connected, has no non-trivial character);
  for instance the subgroup
$$
H=\{(g_1,g_2)\in \Sp_r\times\Sp_r\,\mid\, g_1g_2^{-1}\in \cent(\Sp_r)\}
$$
is a proper subgroup that projects to $\Sp_r$ on both factors; in this
case the representation $\rho_1$ (resp. $\rho_2$) of $H$ obtained by
the first (resp. second) projection satisfies
$$
\rho_2\simeq \rho_1\otimes \chi
$$
where $\chi(g_1,g_2)=g_2g_1^{-1}\in \cent(\Sp_r)\subset \Gg_m$. Thus
$\chi$ is a non-trivial character of $H$.  The same construction works
with $\Sp_r$ replaced by $\SL_r$ in the definition.
\par
(2) This result would not extend if we allow $G_i$ not contained in
$\Gg_m G_i^0$: for instance, if $G=\Ort_{2r}$, so that $G^0=\SO_{2r}$,
there exist irreducible representations of $G$ which split in two
irreducible subrepresentations when restricted to
$G^0$. %%Goodman-Wallach, page 251
\end{remark}

We then state a preliminary result, which for convenience\footnote{\
  See also Remark~\ref{rm-mellin}(1) for suggestions of a
  Mellin-transform analogue of sums of products, where this would be
  the only way to proceed.}  we express in the language of Tannakian
categories. For a $U$-generous tuple $\uple{\sheaf{F}}$, we denote by
$\tanna{\uple{\sheaf{F}}}$ the Tannakian category of sheaves on
$U\times\bFp$ generated by the sheaves $\sheaf{F}_i$.

\begin{proposition}\label{pr-category-2}
  Let $\uple{\sheaf{F}}$ be $U$-generous, and let $\pi\,:\,
  V\times\bFp\ra U\times \bFp$ be the finite abelian \'etale covering
  corresponding to the surjective homomorphism
$$
\pi_1(U\times\bFp,\bar{\eta})\lra G/G^0.
$$
\par
\emph{(1)} The the category $\tanna{\uple{\sheaf{F}}}$ is equivalent
as a Tannakian category to the category of representations of the
linear algebraic group $G$, a functor from the latter to
$\tanna{\uple{\sheaf{F}}}$ giving this equivalence is
$$
\Lambda\mapsto \Lambda\circ \rho_{\uple{\sheaf{F}}}
$$
where $\rho_{\uple{\sheaf{F}}}$ is the representation of
$\pi_1(U\times\bFp,\bar{\eta})$ corresponding to the lisse sheaf
$$
\bigoplus_{i}\sheaf{F}_i.
$$
\par
Furthermore the restriction to $G^0$ of a representation of $G$
corresponds to the functor $\pi^*$.
\par
\emph{(2)} If $\sheaf{G}$ is an irreducible object of
$\tanna{\uple{\sheaf{F}}}$, then we have a geometric isomorphism
$$
\pi^*\sheaf{G}\simeq \bigotimes_{i}\Lambda_i(\pi^*\sheaf{F}_i)
$$
where $\Lambda_i$ is an irreducible representation of $G^0_i$ for each
$i$. Two such sheaves have isomorphic restriction to $V\times\bFp$ if
and only if the respective $\Lambda_i$ are the same.
\end{proposition}

\begin{proof}
  The first part is a standard fact. To deduce (2), we simply note
  that from the last part of Lemma~\ref{lm-gkr}, the pullback
  $\pi^*\sheaf{G}$ is geometrically irreducible if $\sheaf{G}$ is
  geometrically irreducible. We then obtain the stated formula from
  the classification of irreducible representations of a direct
  product.
\end{proof}

We now present a first classification theorem that is well-suited to
cases where all sheaves involved are self-dual.

\begin{theorem}[Diagonal classification]\label{th-diag}
  Let $\uple{\sheaf{F}}$ be $U$-generous and let $\pi\,:\,
  V\times\bFp\ra U\times \bFp$ be the finite abelian \'etale covering
  corresponding to the surjective homomorphism
$$
\pi_1(U\times\bFp,\bar{\eta})\lra G/G^0.
$$
\par
Let $\sheaf{G}$ be an $\ell$-adic sheaf which is geometrically
irreducible and lisse on $U$. Let
$$
\uple{n}=(n_1,\ldots, n_k)
$$
be a $k$-tuple of \emph{positive} integers. Denote
$$
\sheaf{F}_{\uple{n}}=\bigotimes_{1\leq i\leq k} \sheaf{F}_i^{\otimes
  n_i}.
$$
\par
We have
$$
H^2_c(U\times\bFp,\sheaf{F}_{\uple{n}}\otimes\dual(\sheaf{G}))\not=0
$$
only if there exists a geometric isomorphism
\begin{equation}\label{eq-diag-sheaf}
\pi^*\sheaf{G}\simeq \bigotimes_{i}\Lambda_i(\pi^*\sheaf{F}_i)
\end{equation}
on $V\times\bFp$, where, for all $i$, $\Lambda_i$ is an irreducible
representation of the group $G^0_i$ which is also a subrepresentation
of the representation $\std_i^{\otimes n_i}$ of $G^0_i$, with $\std_i$
denoting the natural faithful representation of $G^0_i$ corresponding
to $\pi^*\sheaf{F}_i$.
\par
In fact, for $\sheaf{G}$ given as above, we have
$$
\dim
H^2_c(U\times\bFp,\sheaf{F}_{\uple{n}}\otimes\dual(\sheaf{G}))\leq 
\prod_{1\leq i\leq k} \mathrm{mult}_{\Lambda_i}(\std_i^{\otimes n_i}),
$$ 
where $\mathrm{mult}_{\Lambda_i}(\std_i^{\otimes n_i})$ denotes the
multiplicity of $\Lambda_i$ in $\std_i^{\otimes n_i}$.
\par
If $\uple{\sheaf{F}}$ is strictly $U$-generous, then equality holds in
this formula, and in particular the $H^2_c$ is non-zero if and only if
$\sheaf{G}$ is of the form $\bigotimes_{i}\Lambda_i(\sheaf{F}_i)$ with
$\Lambda_i$ as above.
\par
In general, if $\sheaf{G}$ is of the form~(\ref{eq-diag-sheaf}), then
there exists a character $\chi$ of $G/G^0$ such that
$$
H^2_c(U\times\bFp,
\sheaf{F}_{\uple{n}}\otimes\dual(\sheaf{G}\otimes\chi))\not=0.
$$
\par
If all $n_i$ are equal to $1$, we denote $\sheaf{F}_{(1,\ldots,
  1)}=\sheaf{F}$. Then
$$
\dim H^2_c(U\times\bFp,\sheaf{F}\otimes\dual(\sheaf{G}))=0
$$
unless $\sheaf{G}\simeq \sheaf{F}$, and
$$
\dim H^2_c(U\times\bFp,\sheaf{F}\otimes\dual(\sheaf{G}))=1
$$
in that case.
\end{theorem}

The crucial point in the proof is the following very simple fact:

\begin{lemma}
  With the notation of the theorem, assume that
$$
H^2_c(U\times\bFp,
\sheaf{F}_{\uple{n}}\otimes\dual(\sheaf{G}))\not=0.
$$
\par
Then $\sheaf{G}$ is geometrically isomorphic to an object of
$\tanna{\uple{\sheaf{F}}}$.
\end{lemma}

\begin{proof}
  By the co-invariant formula, the irreducibility of $\sheaf{G}$, and
  the semi-simplicity of the representations involved, the condition
  implies that $\sheaf{G}$ is geometrically isomorphic to a subsheaf
  of $\sheaf{F}_{\uple{n}}$. But clearly this sheaf is itself an
  object of $\tanna{\uple{\sheaf{F}}}$, hence the result by
  transitivity.
\end{proof}

\begin{proof}[Proof of the theorem]
  By the lemma, $\sheaf{G}$ is geometrically isomorphic to an object
  of $\tanna{\uple{\sheaf{F}}}$. Since it is also geometrically
  irreducible, Lemma~\ref{lm-gkr} shows that $\pi^*\sheaf{G}$ is also
  geometrically irreducible. Thus, by the proposition, it follows that
$$
\pi^*\sheaf{G}\simeq \bigotimes_{1\leq i\leq
  k}\Lambda_i(\pi^*\sheaf{F}_i),
$$
where the $\Lambda_i$ are some irreducible representations of the group
$G^0_i$. We have then
$$
\dim H^2_c(U\times\bFp,\sheaf{F}_{\uple{n}}\otimes\dual(\sheaf{G}))
\leq \dim
H^2_c(V\times\bFp,\pi^*\sheaf{F}_{\uple{n}}\otimes\dual(\pi^*\sheaf{G}))
= \dim
(\sheaf{F}_{\uple{n},\bar{\eta}}\otimes\dual(\sheaf{G}_{\bar{\eta}}))^{G^0},
$$
where we can use invariants instead of coinvariants because the
representations are semisimple.  But the $G^0$-invariants of the
generic fibre of
$$
\pi^*\sheaf{F}_{\uple{n}}\otimes\dual(\pi^*\sheaf{G})
=\bigotimes_{1\leq i\leq k}\Bigl(\pi^*\sheaf{F}_i^{\otimes n_i}\otimes
\dual(\Lambda_i(\pi^*\sheaf{F}_i))\Bigr)
$$
are isomorphic (under the equivalence of the proposition) to the
%%tensor product of the respective invariants of $G$ on
invariants of $G^0$ on
$$
\bigboxtimes_{1\leq i\leq k} \Bigl( \std_i^{\otimes n_i}\otimes
\dual(\Lambda_i)\Bigr)
$$
hence to the tensor product over $i$ of the $G^0$-invariants of
$$
\std_i^{\otimes n_i}\otimes \dual(\Lambda_i).
$$
\par
Thus we get the inequality for the dimension, and in particular the
$G^0$-invariant space is non-zero if and only if $\Lambda_i$ is a
subrepresentation of $\std_i^{\otimes n_i}$ for all $1\leq i\leq k$,
and this gives a necessary condition for the $G$-invariant space to be
non-zero.
\par
In the opposite direction, if $\sheaf{G}$ is given
by~(\ref{eq-diag-sheaf}) with $\Lambda_i$ an irreducible
subrepresentation of $\std_i^{\otimes n_i}$, then we have
$$
(\sheaf{F}_{\uple{n},\bar{\eta}}\otimes\dual(\sheaf{G}_{\bar{\eta}}))^{G^0}
\not=0.
$$
\par
This invariant space is naturally a representation of $G/G^0$; since
it is non-zero, it contains at least one character $\chi$; one then
checks easily that
$$
(\sheaf{F}_{\uple{n},\bar{\eta}}\otimes
\dual(\sheaf{G}_{\bar{\eta}}\otimes\chi))^{G}\not=0.
$$
\par
Finally, if $n_i=1$ and the $H^2_c$ is non-zero, then since
$\sheaf{F}$ is irreducible in this case (e.g. because its restriction
to $G^0$ is irreducible as $\bigboxtimes_i \std_i$), Schur's Lemma
gives the result.
%  each
% $\std_i$ is irreducible, and $\Lambda_i$ is a subrepresentation of
% $\std_i$ for all $a$, we must have $\Lambda_i\simeq \std_i$. Then
% $\pi^*\sheaf{G}=\pi^*\sheaf{F}$, and a standard fact from
% representation theory shows that there is a character $\chi$ of
% $G/G^0$ such that
% $$
% \sheaf{G}\simeq \sheaf{F}\otimes \chi.
% $$
% \par
% By Schur's Lemma, the final formula follows.
\end{proof}

\begin{example}
  In the setting of Example~\ref{ex-old}(2), the sheaves
  $\sheaf{F}_{\uple{n}}$ have trace functions
$$
\prod_{1\leq i\leq k} \frfn{\sheaf{F}_0}(a_ix)^{n_i},
$$
and therefore we obtain criteria for square-root cancellation of the
sums
$$
\sum_{x\in \Fpt} \prod_{1\leq i\leq k}
\frfn{\sheaf{F}_0}(a_ix)^{n_i}\frfn{\sheaf{G}}(x).
$$
\par
If we take $\sheaf{G}=\sheaf{L}_{\psi(hX)}$ for some $h$, then we are
in the situation described in the introduction.
\end{example}

We state separately a more general version of Theorem~\ref{th-diag}
which is useful when some sheaves are not self-dual. 

\begin{theorem}[Diagonal classification, 2]\label{th-diag-2}
  Let $\uple{\sheaf{F}}$ be $U$-generous and let $\pi\,:\,
  V\times\bFp\ra U\times \bFp$ be the finite abelian \'etale covering
  corresponding to the surjective homomorphism
$$
\pi_1(U\times\bFp,\bar{\eta})\lra G/G^0.
$$
\par
Let $\sheaf{G}$ be an $\ell$-adic sheaf which is geometrically
irreducible and lisse on $U$. Let
$$
\uple{m}=(m_1,\ldots, m_k),\quad\quad \uple{n}=(n_1,\ldots, n_k)
$$
be $k$-tuples of integers such that $n_i+m_i\geq 1$ for all
$i$. Denote
$$
\sheaf{F}_{\uple{m},\uple{n}}=\bigotimes_{1\leq i\leq k}
\Bigl(\sheaf{F}_i^{\otimes m_i}\otimes \dual(\sheaf{F}_i)^{\otimes
  n_i}\Bigr).
$$
\par
We have
$$
H^2_c(U\times\bFp,\sheaf{F}_{\uple{m},\uple{n}}\otimes\dual(\sheaf{G}))\not=0
$$
only if there exists a geometric isomorphism
\begin{equation}\label{eq-diag-sheaf-2}
\pi^*\sheaf{G}\simeq \bigotimes_{i}\Lambda_i(\pi^*\sheaf{F}_i)
\end{equation}
on $V\times\bFp$, where, for all $i$, $\Lambda_i$ is an irreducible
representation of the group $G^0_i$ which is also a subrepresentation
of the representation $\std_i^{\otimes m_i}\otimes
\dual(\std_i)^{\otimes n_i}$ of $G^0_i$, with $\std_i$ denoting the
natural faithful representation of $G^0_i$ corresponding to
$\pi^*\sheaf{F}_i$.
\par
In fact, for $\sheaf{G}$ given as above, we have
$$
\dim
H^2_c(U\times\bFp,\sheaf{F}_{\uple{m},\uple{n}}\otimes\dual(\sheaf{G}))
\leq 
\prod_{1\leq i\leq k} \mathrm{mult}_{\Lambda_i}(\std_i^{\otimes m_i}\otimes
\dual(\std_i)^{\otimes n_i}),
$$ 
where $ \mathrm{mult}_{\Lambda_i}(\std_i^{\otimes m_i}\otimes
\dual(\std_i)^{\otimes n_i})$ denotes the multiplicity of $\Lambda_i$
in $\std_i^{\otimes m_i}\otimes \dual(\std_i)^{\otimes n_i}$. If
$\uple{\sheaf{F}}$ is strictly $U$-generous, then there is equality,
and the converse also holds.
\par
In general, if $\sheaf{G}$ is given by~(\ref{eq-diag-sheaf-2}), then
there exists a character $\chi$ of $G/G^0$ such that
$$
H^2_c(U\times\bFp,\sheaf{F}_{\uple{m},\uple{n}}
\otimes\dual(\sheaf{G}\otimes\chi))\not=0.
$$
% \par
% If all $n_i$ are equal to $1$, so that
% $\sheaf{F}_{\uple{n}}=\sheaf{F}$, then
% $$
% \dim H^2_c(U\times\bFp,\sheaf{F}\otimes\dual(\sheaf{G}))=0
% $$
% unless $\sheaf{G}\simeq \sheaf{F}$, and
% $$
% \dim H^2_c(U\times\bFp,\sheaf{F}\otimes\dual(\sheaf{G}))=1
% $$
% in that case.
\end{theorem}

Clearly, the case $\uple{n}=(0,\ldots,0)$ recovers
Theorem~\ref{th-diag}. 

\begin{proof}
  This is the same as that of Theorem~\ref{th-diag}, mutatis mutandis.
\end{proof}

Here is a simple corollary that can be very helpful:

\begin{corollary}
  Let $\uple{\sheaf{F}}=(\sheaf{F}_i)_{1\leq i\leq k}$ be
  $U$-generous. Let $\sheaf{G}$ be an $\ell$-adic sheaf. Let
  $\uple{\sigma}$ be a $k$-tuple of elements of $\Aut(\Cc/\Rr)$. If
$$
\rank\sheaf{G}<\prod_i\rank\sheaf{F}_i,
$$
then we have
$$
H^2_c(U\times\bFp,\bigotimes_{1\leq i\leq
  k}\sheaf{F}_i^{\sigma_i}\otimes\dual(\sheaf{G}))=0.
$$
\end{corollary}

\begin{proof}
  Note that this corresponds to the previous situation, with
  $\uple{m}$ and $\uple{n}$ such that $m_i+n_i=1$ for all $i$.
\par
By considering a geometrically irreducible subsheaf of $\sheaf{G}$, we
may assume that it is geometrically irreducible (since a subsheaf
still satisfies the dimension bound and $H^2_c$ is additive).  By the
previous arguments, if the $H^2_c$ were non-zero, then we would then
have
$$
\pi^*\sheaf{G}\simeq \bigotimes_i \Lambda_i(\pi^*\sheaf{F}_i),
$$
where $\Lambda_i$ is irreducible and occurs in $\std_{i}$. But this
implies that $\Lambda_i\simeq \std_i$, and in particular that
$$
\rank\sheaf{G}=\prod_i \rank\sheaf{F}_i.
$$
\end{proof}

We will use the following additional lemma in
Section~\ref{sec-control}:

\begin{lemma}\label{lm-intersect}
  Let $\uple{\sheaf{F}}_1=(\sheaf{F}_{1,i})$ and
  $\uple{\sheaf{F}}_2=(\sheaf{F}_{2,j})$ be tuples of sheaves.
\par
Let $\uple{\sheaf{F}}_3$ be the tuple containing those sheaves which
occur, up to geometric isomorphism, in both $\uple{\sheaf{F}}_1$ and
$\uple{\sheaf{F}}_2$, and let $\uple{\sheaf{F}}_4$ be the tuple
containing those sheaves which occur in either $\uple{\sheaf{F}}_1$ or
$\uple{\sheaf{F}}_2$. Assume that $\uple{\sheaf{F}}_4$ is
$U$-generous, and let $\pi\,:\, V\times\bFp\ra U\times \bFp$ be the
finite abelian \'etale covering corresponding to the surjective
homomorphism
$$
\pi_1(U\times\bFp,\bar{\eta})\lra G/G^0
$$
corresponding to this generous tuple.
\par
Let $\sheaf{G}$ be an $\ell$-adic sheaf on $U$ which is geometrically
isomorphic both to some object in $\tanna{\uple{\sheaf{F}}_1}$ and to
some object in $\tanna{\uple{\sheaf{F}}_2}$. Then $\pi^*\sheaf{G}$ is
geometrically isomorphic to $\pi^*\sheaf{G}_1$ for some object
$\sheaf{G}_1$ in $\tanna{\uple{\sheaf{F}}_3}$.
\end{lemma}

\begin{proof}
We denote by $G_{1,i}$ (resp. $G_{2,j}$) the geometric monodromy
groups of the $\sheaf{F}_{1,i}$ (resp. $\sheaf{F}_{2,j}$). Let
$$
\rho\,:\, \pi_1(U\times\bFp,\bar{\eta})\lra \GL(W)
$$
be the $\ell$-adic representation corresponding to the sheaf
$$
\bigoplus_{i}\sheaf{F}_{1,i}\oplus \bigoplus_{j}\sheaf{F}_{2,j},
$$
and let $G$ be its geometric monodromy group, which is a subgroup of
$$
\prod_{i}G_{1,i}\times \prod_j G_{2,j}.
$$
\par
The objects of $\tanna{\uple{\sheaf{F}}_1}$
(resp. $\tanna{\uple{\sheaf{F}}_2}$) are those objects of
$\tanna{\uple{\sheaf{F}}_4}$ which correspond to representations of
$G$ trivial on
$$
G\cap \prod_j G_{2,j}\quad\quad
\text{(resp. trivial on $G\cap \prod_i G_{1,i}$)}.
$$
\par
Consequently, objects belonging to both $\tanna{\uple{\sheaf{F}}_1}$
and $\tanna{\uple{\sheaf{F}}_2}$ are representations of $G$ trivial on
$$
(G\cap \prod_j G_{1,i})\times (G\cap \prod_i G_{2,j}).
$$
\par
On the other hand, for $I'\subset I$ and $J'\subset J$ parameterizing
the tuple $\uple{\sheaf{F}}_3$, and $\sigma\,:\, I'\lra J'$ a
bijection such that $\sheaf{F}_{1,i}$ and $\sheaf{F}_{2,\sigma(i)}$
are geometrically isomorphic, the objects of
$\tanna{\uple{\sheaf{F}}_3}$ correspond to representations of the
geometric monodromy group $G'$ of
$$
\bigoplus_{i\in I'}{(\sheaf{F}_{1,i}\oplus \sheaf{F}_{2,\sigma(i)})}.
$$
\par
This can be identified with the group $G\cap H$, where $H$ is the
subgroup of
$$
\prod_{i}G_{1,i}\times \prod_j G_{2,j}
$$
with coordinates $(x_i)_{i\in I}$, $(y_j)_{j\in J}$, determined by the
conditions $x_i=1$ for $i\notin I'$, $y_j=1$ for $j\notin J'$, and
$$
y_{\sigma(i)}=\alpha_ix_i\alpha_i^{-1}
$$
for all $i\in I'$, where $\alpha_i$ is fixed (the inner automorphism
by $\alpha_i$ realizing the geometric isomorphism of
$\sheaf{F}_{1,i}$ with $\sheaf{F}_{2,\sigma(i)}$.) 
\par
The analogue assertions hold after pullback under $\pi$, if all
$G_{1,i}$ and $G_{2,j}$ are replaced with their respective connected
components.
\par
By the assumption that $\uple{\sheaf{F}}_4$ is $U$-generous, we see
that $G^0$ is equal to the product
$$
\{(x,y,\alpha(y),z)\,\mid\, x\in \prod_{i\in I-I'}G_{1,i}^0\,\ 
y\in\prod_{i\in I'}{G_{1,i}},\ 
z\in \prod_{J-\sigma(I')}G_{2,j}^0\}\subset G
$$
(where $\alpha$ is the isomorphism
$$
\prod_{i\in I'}{G_{1,i}}\lra \prod_{j\in J'}{G_{2,j}}
$$
given by mapping $x_i\in G_{1,i}$ to $\alpha_ix_i\alpha_i^{-1}\in
G_{2,\sigma(i)}$) and therefore we find that
$$
G^0/(G^0\cap \prod_i G_{1,i}^0)\times (G^0\cap \prod_j
G_{2,j}^0)\simeq G^0\cap H.
$$
\par
This gives the desired conclusion.
\end{proof}

\section{Examples}\label{sec-examples}

We collect here examples of trace functions for which the results
stated in the introduction or in the previous section apply, and state
some of the resulting bounds for convenience. These examples are
taken for the most part from the many results of Katz, who has
computed the monodromy groups of many classes of sheaves over $\Aa^1$
using a variety of techniques.
\par

\subsection{General construction}\label{ssec-general}
Quite generally, let $(\sheaf{F}_i)_{i \in I}$ be any finite tuple of
middle-extension sheaves of weight $0$ on $\Aa^1_{\Fp}$ such that the
geometric monodromy groups $G_i$ of the restriction of $\sheaf{F}_i$
to a dense open set $U_i$ where it is lisse, is such that $G_i^0$ is
any of the groups
\begin{gather*}
  \SL_r,\text{ for $r\geq
    3$},\quad \quad \SO_{2r+1},\text{ for $r\geq 1$},\\
  \Sp_{r},\text{ for $r$ even $\geq 2$},\\
%%  \SO_{2r},\text{ for $r\geq 3$, $r\not=4$}\\
  \mathbf{F}_4,\quad \mathbf{E}_7,\quad \mathbf{E}_8,\quad
  \mathbf{G}_2.
\end{gather*}
\par
Then we can always extract a convenient generous subtuple as follows:
let $U$ be the intersection of the $U_i$, and let $J\subset I$ be any
set of representatives of $I$ for the equivalence relation defined by
$i\sim j$ if and only if
$$
\sheaf{F}_i\simeq \sheaf{F}_j\otimes \sheaf{L},\text{ or }
\dual(\sheaf{F}_i)\simeq \sheaf{F}_j\otimes\sheaf{L}
$$
on $U$ for some rank $1$ sheaf $\sheaf{L}$ lisse on $U$. Then
$\uple{\sheaf{F}}=(\sheaf{F}_i)_{i\in J}$ is $U$-generous.
\par
Indeed, condition (1) is clear, and (2) holds by the restrictions on
$G_i^0$ (see also~\cite[9.3.6]{katz-sarnak} for the normalizer
condition, and note that in the exceptional cases indicated, all
automorphisms of the groups are inner, which implies the normalizer
condition). Also, by~\cite[Examples 1.8.1]{katz-esde}, the
representations corresponding to $i\not=j$ in $J$ are always
Goursat-adapted, and finally the restriction to the representatives of
the equivalence relation ensures the last condition.
\par
Note that for any multiplicities $n_i$, $m_i\geq 0$ for $i\in I$, we
have then geometric isomorphisms
$$
\bigotimes_{i\in I} \sheaf{F}_i^{\otimes n_i}
\otimes
\bigotimes_{i\in I} \dual(\sheaf{F}_i)^{\otimes m_i}
\simeq 
\sheaf{L}
\bigotimes_{i\in J} \sheaf{F}_i^{\otimes n'_i}
\otimes
\bigotimes_{i\in J} \dual(\sheaf{F}_i)^{\otimes m'_i}
$$
for some rank $1$ sheaf $\sheaf{L}$ (depending on $(n_i,m_i)$) and
$$
n'_i=\sum_{j\sim i}{n_j},\quad\quad
m'_i=\sum_{j\sim i}{m_j},
$$
and it is therefore possible to use many of the results for the
generous tuple $\uple{\sheaf{F}}$ to derive corresponding statements
that apply to the original one. For an example of applying this
principle, see the discussion of the Bombieri--Bourgain sums in
Section~\ref{sec-applications}.
\par
In applications of this strategy, especially in the $\SL_r$ case, the
following lemma will be useful:

\begin{lemma}\label{lm-special-coset}
  Let $\sheaf{F}$ be an $\ell$-adic sheaf modulo $p$. Then
  $\Autt(\sheaf{F})$ is either empty or is of the form
  $\xi\Autz(\sheaf{F})$ for some $\xi\in N(\Autz(\sheaf{F}))$ such
  that $\xi^2\in \Autz(\sheaf{F})$.
\end{lemma}

For $\Autz(\sheaf{F})=1$, we recover the fact that $\Autt(\sheaf{F})$
is either empty or contains only an involution; if $\Autz(\sheaf{F})$
is equal to its normalizer, e.g., if it is a maximal and non-normal
subgroup, then it shows that $\Autt(\sheaf{F})$ is either empty or
equal to $\Autz(\sheaf{F})$, which means that $1\in\Autt(\sheaf{F})$,
or in other words that
$$
\sheaf{F}\simeq \dual(\sheaf{F})\otimes\sheaf{L}
$$
for some rank $1$ sheaf $\sheaf{L}$.  This means that, in some sense,
$\sheaf{F}$ is ``almost'' self-dual.

\begin{proof}
  More generally, consider a subgroup $H$ of a group $G$, and a coset
  $T\subset G$ of the form $T=\xi H$ that satisfies $g^2\in G$ for all
  $g\in T$ (as is the case of $T=\Autt(\sheaf{F})\subset
  G=\PGL_2(\bFp)$ for the subgroup $H=\Autz(\sheaf{F})$). 
\par
We claim first that this situation occurs if and only if $T=\xi H$ for
some $\xi\in G$ such that $\xi H\xi =H$.
\par
Indeed, $(\xi g)( \xi g)\in H$ for all $g\in H$ is equivalent to $\xi
g\xi\in H$ for all $g\in H$, i.e., to $\xi H\xi\subset H$. But then
the converse inclusion $\xi H\xi\supset H$ also holds by taking the
inverse: 
$$
\xi^{-1}H\xi^{-1}=(\xi H\xi)^{-1}\subset H^{-1}=H.
$$
\par
Now from $\xi H\xi =H$, we get first in particular $\xi^2\in H$, and
then
$$
H=\xi H\xi =\xi (H\xi^2) \xi^{-1}=\xi H \xi^{-1}
$$
implies that $\xi\in N(H)$. This gives the result in our case, and we
may also note that the converse holds, namely if $\xi\in N(H)$
satisfies $\xi^2\in H$, then
$$
\xi H\xi =\xi H \xi^2 \xi^{-1}=\xi H\xi^{-1}=H.
$$
\end{proof}

\begin{remark}
  It is amusing to note that $\xi H\xi \subset H$ implies that $\xi
  H\xi =H$, whereas $\xi H\xi^{-1}\subset H$ does not, in general,
  imply that $\xi H\xi^{-1}=H$ (see~\cite[A I, p. 134,
  Ex. 27]{bourb-alg} for a counterexample). One can show that, for
  arbitrary $(a,b)\in\Zz^2$ with $a+b\not=0$, the condition $\xi^a
  H\xi^b\subset H$, for a subgroup $H\subset G$ and an element $\xi\in
  G$, always implies $\xi^a H \xi^b=H$.
\end{remark}

%%http://blogs.ethz.ch/kowalski/2014/05/03/more-conjugation-shenanigans/

Looking at the list of simple groups at the beginning of this section,
it is clear that the only significant omission is that of
$G_i^0=\SO_{2r}$ for $r\geq 2$; in that case, it is indeed not true
that the normalizer $O_{2r}$ is contained in $\Gg_m G_i^0$ (see also
Remark~\ref{rm-so} (2) below). This complication may be problematic in
some applications, since geometric monodromy groups $\Ort_{2r}$ do
occur naturally (e.g., for certain hypergeometric sheaves and for
elliptic curves over function fields, see
Section~\ref{sec-examples}). However, we have not (yet) encountered
such cases in analytic number theory, and one can expect that some
analogues of our statements could be proved using the classification
of representations of $\Ort_{2r}$ and their restrictions to
$\SO_{2r}$.

\subsection{Even rank Kloosterman sums} For $r\geq 2$ even, the normalized Kloosterman sums
$$
\hypk_r(x;p)=-\frac{1}{p^{(r-1)/2}} \sum_{t_1\cdots t_r=x}
  e\Bigl(\frac{t_1+\cdots+t_r}{p}\Bigr)
$$
are the trace functions of a self-dual bountiful sheaf $\HYPK_r$ on
$\Aa^1_{\Fp}$ with conductor uniformly bounded for all $p$. Indeed,
the geometric monodromy group is then $\Sp_r$
by~\cite[Th. 11.1]{katz-gkm}, and the projective automorphism group is
trivial by Proposition~\ref{pr-aut-hyper} below.  In addition, one
knows that the arithmetic monodromy group of $\HYPK_r$ is equal to its
geometric monodromy group, so that Corollary~\ref{cor-concrete2}
applies to this sheaf.
\par
Hence, from Corollary~\ref{cor-concrete}, we get:

\begin{corollary}\label{cor-even-kloos}
  Let $r\geq 2$ be an even integer. Let $k\geq 1$ be an integer. There
  exists a constant $C\geq 1$, depending only on $k$ and $r$ such that
  for any prime $p$, any $h\in\Fp$ and any
  $\uple{\gamma}=(\gamma_1,\ldots,\gamma_k)\in \PGL_2(\Fp)$ and
  $h\in\Fp$, such that either
\begin{itemize}
\item we have $h\not=0$, or;
\item some component of
  $\uple{\gamma}$ occurs with odd multiplicity, i.e., $\uple{\gamma}$
  is normal, as in Definition~\ref{def-normal}.
\end{itemize}
\par
Then we have
$$
\Bigl|\sums_{x\in\Fp}
\hypk_r(\gamma_1\cdot x;p)\cdots \hypk_r(\gamma_k\cdot x;p)
e\Bigl(\frac{hx}{p}\Bigr)\Bigr|
\leq Cp^{1/2}
$$
where the sum runs over $x$ such that all $\gamma_i\cdot x$ are
defined.
\end{corollary}

\subsection{Odd rank Kloosterman sums}
For $r\geq 2$ odd, the normalized Kloosterman sums
$$
\hypk_r(x;p)=\frac{1}{p^{(r-1)/2}} \sum_{t_1\cdots t_r=x}
  e\Bigl(\frac{t_1+\cdots+t_r}{p}\Bigr)
$$
are the trace functions of a non-self-dual bountiful sheaf $\HYPK_r$
on $\Aa^1_{\Fp}$ of $\SL_r$ type, with conductor uniformly bounded
over $p$, with special involution $x\mapsto -x$.  Indeed, the
geometric monodromy group is $\SL_r$ by~\cite[Th. 11.1]{katz-gkm}, and
the projective automorphism group is trivial by
Proposition~\ref{pr-aut-hyper} below, and we also have a geometric
isomorphism
$$
\dual(\HYPK_r)\simeq [\times (-1)]^*\HYPK_r.
$$
\par
In addition, one knows that the arithmetic monodromy group of
$\HYPK_r$ is equal to its geometric monodromy group, and hence
Corollary~\ref{cor-concrete2} also applies to this sheaf of
$\SL_r$-type.
% so that one may (in that case) ``remove'' the duals from any tensor
% product of pullbacks of $\HYPK_r$.
\par
Hence, from Corollary~\ref{cor-concrete}, we get:

\begin{corollary}\label{cor-odd-kloos}
  Let $r\geq 2$ be an odd integer. Let $k\geq 1$ be an integer. There
  exists a constant $C\geq 1$, depending only on $k$ and $r$ such that
  for any prime $p$, any $h\in\Fp$ and any
  $\uple{\gamma}=(\gamma_1,\ldots,\gamma_k)\in \PGL_2(\Fp)^k$ and
  $\uple{\sigma}=(\sigma_1,\ldots,\sigma_k)\in \Aut(\Cc/\Rr)^k$, such
  that either
\begin{itemize}
\item we have $h\not=0$, or;
\item the pair $(\uple{\gamma},\uple{\sigma})$ is $r$-normal with
  respect to $x\mapsto -x$.
\end{itemize} 
\par
Then we have
$$
\Bigl|\sums_{x\in\Fp} \hypk_r(\gamma_1\cdot x;p)^{\sigma_1}\cdots
\hypk_r(\gamma_k\cdot x;p)^{\sigma_k} 
e\Bigl(\frac{hx}{p}\Bigr)\Bigr| \leq Cp^{1/2}
$$
where the sum runs over $x$ such that all $\gamma_i\cdot x$ are
defined.
\end{corollary}

Concretely, recall (see~(\ref{eq-r-normal}) and the examples
following) that to say that the pair $(\uple{\gamma},\uple{\sigma})$
is $r$-normal with respect to $x\mapsto -x$ means that for \emph{some}
component $\gamma$ of $\uple{\gamma}$, we have
$$
r\nmid (a_1+a_2)-(b_1+b_2),
$$
where:
\begin{itemize}
\item $a_1$ is the number of $i$ with $\gamma=\gamma_i$ and
  $\sigma_i=1$
\item $a_2$ is the number of $i$ with $\gamma=\begin{pmatrix}-1&0\\0&1
  \end{pmatrix}\gamma_i$ and $\sigma_i=c$
\item $b_1$ is the number of $i$ with $\gamma=\gamma_i$ and
  $\sigma_i=c$
\item $b_2$ is the number of $i$ with $\gamma=\begin{pmatrix}-1&0\\0&1
  \end{pmatrix}\gamma_i$ and $\sigma_i=1$.
\end{itemize}
% \par
% For instance, when $r=3$, $k=6$, and
% \begin{gather*}
% \uple{\gamma}=\Bigl(1,1,\Bigr)\\
% \uple{\sigma}=\Bigl(...\Bigr)
% \end{gather*}
% the pair $(\uple{\gamma},\uple{\sigma})$ is not $r$-normal with
% respect to $x\mapsto -x$.

\subsection{Hypergeometric sums}

Hyper-Kloosterman sums have been generalized by
Katz~\cite[Ch. 8]{katz-esde} to hypergeometric sums, which are
analogues of general hypergeometric functions. Some give rise to
bountiful sheaves, and many to generous tuples. We recall the
definition: given a prime number $p$, integers $m$, $n\geq 1$, with
$m+n\geq 1$, and tuples $\uple{\chi}=(\chi_i)_{1\leq i\leq n}$ and
$\uple{\rho}=(\rho_j)_{1\leq j\leq m}$ of multiplicative characters of
$\Fpt$, the hypergeometric sum $\hypg(\uple{\chi},\uple{\rho},t;p)$ is
defined (see~\cite[8.2.7]{katz-esde}) for $t\in\Fp$ by
$$
\hypg(\uple{\chi},\uple{\rho},t;p)= \frac{(-1)^{n+m-1}}{p^{(n+m-1)/2}}
\sum_{N(\uple{x})=tN(\uple{y})}
\prod_{i}\chi_i(x_i)\overline{\prod_{j}\rho_j(y_j)}
e\Bigl(\frac{T(\uple{x})-T(\uple{y})}{p}\Bigr)
$$
where
\begin{gather*}
  N(\uple{x})=x_1\cdots x_n,\quad\quad N(\uple{y})=y_1\cdots y_m,\\
  T(\uple{x})=x_1+\cdots +x_n,\quad\quad T(\uple{y})=y_1+\cdots +y_m
\end{gather*}
so that the sum is over all $(n+m)$-tuples
$(\uple{x},\uple{y})\in\Fp^{n+m}$ such that
$$
x_1\cdots x_n=ty_1\cdots y_m.
$$
\par
If $n=r$, $m=0$, and $\chi_i=1$ for all $i$, then we recover the
Kloosterman sums $\hypk_r(t;p)$. If $n=2$, $m=0$, and $\chi_2=1$ but
$\chi_1$ is non-trivial, we obtain Sali\'e-type sums. This indicates
that such sums should arise naturally in formulas like the Voronoi
summation formula for automorphic forms with non-trivial nebentypus.
\par
Katz shows (see~\cite[Th. 8.4.2]{katz-esde}) that if no character
$\chi_i$ coincides with a character $\rho_j$ (in which case one says
that $\uple{\chi}$ and $\uple{\rho}$ are \emph{disjoint}), then for
any $\ell\not=p$, there exists an irreducible $\ell$-adic
middle-extension sheaf $\HYPG(\uple{\chi},\uple{\rho})$ on
$\Aa^1_{\Fp}$, of weight $0$, with trace function given by
$\hypg(\uple{\chi},\uple{\rho},t;p)$. This sheaf is lisse on $\Gg_m$,
except if $m=n$, in which case it is lisse on $\Gg_m-\{1\}$. It has
rank $\max(m,n)$. Moreover, the conductor of
$\HYPG(\uple{\chi},\uple{\rho})$ is bounded in terms of $m$ and $n$
only.
\par
The basic results of Katz concerning the geometric monodromy group $G$
of the hypergeometric sheaf $\HYPG(\uple{\chi},\uple{\rho})$ depend on
the following definitions of exceptional tuples of characters
(see~\cite[Cor. 8.9.2, 8.10.1]{katz-esde}):

\begin{definition}
Let $k$ be a finite field and let $\uple{\chi}$ and $\uple{\rho}$ be
an $n$-tuple and an $m$-tuple of characters of $k^{\times}$. 
\par
(1) For $d\geq 1$, the pair $(\uple{\chi},\uple{\rho})$ is
$d$-Kummer-induced if $d\mid (n,m)$ and if there exist $n/d$ and
$m/d$-tuples $\uple{\chi}^*$ and $\uple{\rho}^*$ such that
$\uple{\chi}$ consists of all characters $\chi$ such that $\chi^d$ is
a component of $\uple{\chi}^*$, and $\uple{\rho}$ consists of all
characters $\rho$ such that $\rho^d$ is a component of
$\uple{\rho}^*$.
\par
(2) Assume $n=m$. For integers $a$, $b\geq 1$ such that $a+b=n$, the
pair $(\uple{\chi},\uple{\rho})$ is $(a,b)$-Belyi-induced if there
exist characters $\alpha$ and $\beta$ with $\beta\not=1$ such that 
$\uple{\chi}$ consists of all characters $\chi$ such that either
$\chi^a=\alpha$ or $\chi^b=\beta$, and if $\uple{\rho}$ consists of
all characters $\rho$ such that $\rho^{n}=\alpha\beta$.
\par
(3) Assume $n=m$. For integers $a$, $b\geq 1$ such that $a+b=n$, the
pair $(\uple{\chi},\uple{\rho})$ is $(a,b)$-inverse-Belyi-induced if
and only if $(\overline{\uple{\rho}},\overline{\uple{\chi}})$ is
$(a,b)$-Belyi-induced.
\end{definition}

We say that $(\uple{\chi},\uple{\rho})$ is Kummer-induced
(resp. Belyi-induced, inverse-Belyi-induced) if there exists some
$d\geq 2$ (resp. some $a$, $b\geq 1$) such that the pair is
$d$-Kummer-induced (resp. $(a,b)$-Belyi-induced,
$(a,b)$-inverse-Belyi-induced).

We then have the following:
\begin{itemize}
\item If $n=m$, let $\Lambda$ denote the multiplicative character
$$
\Lambda=\prod_{i}\chi_i\overline{\rho_i}.
$$
\par
Assume that $(\uple{\chi},\uple{\rho})$ is neither Kummer-induced,
Belyi-induced, nor inverse-Belyi-induced. Then $G^0$ is either
trivial, $\SL_n$, $\SO_n$ or $\Sp_n$; if $\Lambda=1$, it is either
$\SL_n$ or $\Sp_n$, if $\Lambda\not=1$ but $\Lambda^2=1$, then $G^0$
is either $1$ or $\SO_n$ or $\SL_n$, and if $\Lambda^2\not=1$, then
$G^0$ is either $1$ or $\SL_n$ (see~\cite[Th. 8.11.2]{katz-esde}). The
problem of determining which case occurs is discussed by Katz; most
intricate is the criterion for $G^0$ to be trivial (see~\cite[\S
8.14--8.17]{katz-esde}), which is however applicable in practice.
\item If $n\not=m$, let $r=\max(n,m)$ be the rank of the sheaf. Assume
  that $(\uple{\chi},\uple{\rho})$ is not Kummer induced. Then,
  provided $p>2\max(n,m)+1$, and $p$ does not divide an explicit
  positive integer, we have: $G^0=\SL_r$ if $n-m$ is odd (and
  $G\not=G^0$ if $|n-m|=1$); $G^0=\SL_r$, $\SO_r$ or $\Sp_r$ if $n-m$
  is even and either $r\notin \{7,8,9\}$ or $|n-m|\not=6$
  (see~\cite[Th. 8.11.3]{katz-esde}). Here also, more precise criteria
  for which $G^0$ arises exist, as well as a classification of the few
  exceptional possibilities when $|n-m|=6$ and $r\in\{6,7,8\}$.
\end{itemize}

\begin{example}
  If $\uple{\rho}$ is the empty tuple, $n\geq 2$ and $\uple{\chi}$ is
  an $n$-tuple where all components are trivial, then it follows
  immediately from the definition that $(\uple{\chi},\uple{\rho})$ is
  not Kummer-induced. Thus the last result recovers, for $p$ large
  enough in terms of $n$, the fact that the geometric monodromy group
  of $\HYPK_n$ contains $\SL_n$ if $n$ is odd, and contains either
  $\SO_n$ or $\Sp_n$ if $n$ is even. 
\end{example}

In order to apply the results of the previous section, it is of course
very useful to have some information concerning the projective
automorphism groups of hypergeometric sheaves.  Many cases are
contained in the following result:

\begin{proposition}\label{pr-aut-hyper}
  \emph{(1)} Let $\uple{\chi}_1$, $\uple{\rho}_1$ and $\uple{\chi}_2$,
  $\uple{\rho}_2$ be any $n_1$-tuple (resp. $m_1$-tuple, $n_2$-tuple,
  $m_2$-tuple) with $\uple{\chi}_1$ disjoint from $\uple{\rho}_1$ and
  $\uple{\chi}_2$ disjoint from $\uple{\rho}_2$, and with $m_1+n_1\geq
  1$, $m_2+n_2\geq 1$. Let $a\in\Fpt$. Then we have a geometric
  isomorphism
\begin{equation}\label{eq-hyp-iso}
[\times a]^*\HYPG(\uple{\chi}_1,\uple{\rho}_1)\simeq
\HYPG(\uple{\chi}_2,\uple{\rho}_2),
\end{equation}
if and only if $a=1$ and $\uple{\chi}_1\sim\uple{\chi}_2$ and
$\uple{\rho}_1\sim \uple{\rho}_2$.
\par
\emph{(2)} Let $m\not=n$ with $m+n\geq 1$ be integers with
$\max(m,n)\geq 2$ and $(m,n)\not=(1,2)$, $(m,n)\not=(2,1)$. Let
$\uple{\chi}$ and $\uple{\rho}$ be disjoint tuples of characters of
$\Fpt$. The projective automorphism group
$\Autz(\HYPG(\uple{\chi},\uple{\rho}))$ is then trivial.
\par
\emph{(3)} With notation as in \emph{(2)}, the set
$\Autt(\HYPG(\uple{\chi},\uple{\rho}))$ is non-empty if and only if
the integer $n-m$ is odd, and the tuples $\uple{\chi}$ and
$\uple{\rho}$ are both invariant under inversion. In this case, the
special involution is $x\mapsto -x$, i.e., we have
$$ [\times
(-1)]^*\HYPG(\uple{\chi},\uple{\rho})\simeq
\dual(\HYPG(\uple{\chi},\uple{\rho})).
$$
\par
\emph{(4)} If $n=m\geq 2$, then for any disjoint $n$-tuples
$(\uple{\chi},\uple{\rho})$, the group
$\Autz(\HYPG(\uple{\chi},\uple{\rho}))$ is a subgroup of the finite
group
$$
\Gamma=\Bigl\{
1,
\begin{pmatrix}
0&1\\1&0
\end{pmatrix},
\begin{pmatrix}
-1&1\\0&1
\end{pmatrix},
\begin{pmatrix}
0&1\\-1&1
\end{pmatrix},
\begin{pmatrix}
1&0\\1&-1
\end{pmatrix},
\begin{pmatrix}
1&-1\\1&0
\end{pmatrix}
\Bigr\}\subset \PGL_2(\bFp).
$$
\par
\emph{(5)} With notation as in~\emph{(4)}, the set
$\Autt(\HYPG(\uple{\chi},\uple{\rho}))$ is either empty or is a subset
of $\Gamma$, which is of the form $T=\xi H$ for some subgroup
$H\subset \Gamma$ and some $\xi\in N_{\Gamma}(H)$ such that $\xi^2\in
H$.
\end{proposition}

\begin{proof}
  (1) For all $a\in\bFp^{\times}$, the components of $\uple{\chi}$
  (resp. $\uple{\rho}$) can be recovered from the sheaf $[\times
  a]^*\HYPG(\uple{\chi},\uple{\rho})$ as the tame characters occuring
  in the representation of the inertia group at $0$ (resp. at
  $\infty$) corresponding to this sheaf, and the multiplicity appears
  as the size of the associated Jordan block (see~\cite[Th. 8.4.2 (6),
  (7), (8)]{katz-esde}). Thus~(\ref{eq-hyp-iso}) is only possible if
  $\uple{\chi}_1\sim\uple{\chi}_2$ and $\uple{\rho}_1\sim
  \uple{\rho}_2$.
\par
We assume this is the case now, i.e., that 
%  and wish to prove furthermore that we
% must have $a=1$. The easiest argument is to observe that a geometric
% isomorphism
$$
[\times a]^*\HYPG(\uple{\chi},\uple{\rho})\simeq 
\HYPG(\uple{\chi},\uple{\rho}).
$$
\par
We then obtain $a=1$ from~\cite[Lemma 8.5.4]{katz-esde} and the fact
that the Euler-Poincar\'e characteristic of a hypergeometric sheaf is
$-1$.
\par
(2) We may assume that $n>m$, using inversion otherwise. Assume that
$\gamma\in\PGL_2(\bFp)$ is such that
$$
\gamma^*\HYPG(\uple{\chi},\uple{\rho})\simeq
\HYPG(\uple{\chi},\uple{\rho}) \otimes\sheaf{L}
$$
for a rank $1$ sheaf $\sheaf{L}$.  By comparing ramification behavior
we see that $\gamma$ must be diagonal (if $\gamma^{-1}(0)\not=0$, then
$\sheaf{L}$ must be tamely ramified at $0$ to have the tensor product
tamely ramified at $\gamma^{-1}(0)$, as
$\gamma^*\HYPG(\uple{\chi},\uple{\rho})$ is; but then the inertia
invariants at $\gamma^{-1}(0)$ are zero for the tensor product, a
contradiction to~\cite[Th. 8.4.2 (6)]{katz-esde}, and the case of
$\gamma^{-1}(\infty)\not=\infty$ gives a similar contradiction).
% If
% it is anti-diagonal, we have
% $$
% \gamma\cdot x=\frac{a}{x}
% $$
% for some $a\in\Fpt$. We then have
% \begin{align*}
%   \gamma^*\HYPG(\uple{\chi},\uple{\rho}) &\simeq [\times
%   a^{-1}]^*[x\mapsto 1/x]^*\HYPG(\uple{\chi},\uple{\rho})
%   \\
%   &\simeq [\times a^{-1}]^*\HYPG_{\bar{\psi}}(\overline{\uple{\rho}},
%   \overline{\uple{\chi}})\\
%   &\simeq [\times b]^*\HYPG(\overline{\uple{\rho}},
%   \overline{\uple{\chi}})
% \end{align*}
% where $b=(-1)^{n-m}a^{-1}$, by combining~\cite[8.3.3]{katz-esde}
% and~\cite[Lemma 8.7.2]{katz-esde}.
\par
Thus $\gamma\in\Autz(\HYPG(\uple{\chi},\uple{\rho}))$ implies a
geometric isomorphism
$$
[\times a]^*\HYPG(\uple{\chi},\uple{\rho}) \simeq
\HYPG(\uple{\chi},\uple{\rho})\otimes\sheaf{L}
$$
on some dense open set $j\,:\, U\injecte \Gg_m$.  By~\cite[Lemma
8.11.7.1]{katz-esde}, under the current assumption $(n,m)\not=(2,1)$,
this implies that $\sheaf{L}\simeq \sheaf{L}_{\Lambda}$ for some
multiplicative character $\Lambda$.
\par
But then we have
$$
\HYPG(\uple{\chi},\uple{\rho})\otimes\sheaf{L}_{\Lambda}\simeq
\HYPG(\Lambda\uple{\chi},\Lambda\uple{\rho})
$$
by~\cite[8.3.3]{katz-esde} where
$\Lambda\uple{\chi}=(\Lambda\chi_i)_i$ and
$\Lambda\uple{\rho}=(\Lambda\rho_j)_j$. We are therefore reduced to a
geometric isomorphism
$$
[\times a]^*\HYPG(\uple{\chi},\uple{\rho})\simeq
\HYPG(\Lambda\uple{\chi},\Lambda\uple{\rho}),
$$
and by (1), it follows that $a=1$, i.e., $\gamma=1$.
%  If $\gamma$ is
% anti-diagonal, even the condition
% $$
% \overline{\uple{\rho}}\sim\uple{\chi}
% $$
% gives a contradiction since $m\not=n$.
\par
(3) As in the previous case, we see that any element $\gamma\in
\Autt(\HYPG(\uple{\chi},\uple{\rho}))$ must be diagonal, so that
$\gamma\cdot x=ax$ for some $a\in\Fpt$. Since $\gamma$, if it exists,
is an involution, we obtain $a^2=1$, and therefore the only
possibility for the special involution is $x\mapsto -x$.
\par
We now assume that
$$
[\times (-1)]^*\HYPG(\uple{\chi},\uple{\rho})\simeq
\dual(\HYPG(\uple{\chi},\uple{\rho}))\otimes\sheaf{L}
$$
for some rank $1$ sheaf $\sheaf{L}$.
\par
Again by~\cite[Lemma 8.11.7.1]{katz-esde}, the sheaf $\sheaf{L}$ is a
Kummer sheaf $\sheaf{L}_{\Lambda}$. We have
$$
\dual(\HYPG(\uple{\chi},\uple{\rho}))\otimes\sheaf{L} \simeq
\HYPG_{\bar{\psi}}(\overline{\uple{\chi}},\overline{\uple{\rho}})
\otimes \sheaf{L}_{\Lambda} \otimes [\times (-1)^{n-m}]^*
\HYPG(\Lambda\overline{\uple{\chi}},\Lambda\overline{\uple{\rho}})
$$
by combining~\cite[8.3.3]{katz-esde} and~\cite[Lemma 8.7.2]{katz-esde}
(using also the fact that Kummer sheaves are geometrically
multiplication invariant). Thus the assumption means that
$$
[\times (-1)]^*\HYPG(\uple{\chi},\uple{\rho})\simeq [\times
(-1)^{n-m}]^*\HYPG(\Lambda\overline{\uple{\chi}},
\Lambda\overline{\uple{\rho}}).
$$
\par
If $n-m$ is even, this can not happen by (1); if $n-m$ is odd, on the
other hand, this happens if and only if
$\Lambda\overline{\uple{\chi}}\sim\uple{\chi}$ and
$\Lambda\overline{\uple{\rho}}\sim\uple{\rho}$, as claimed.
\par
(4) Let $n=m\geq 2$ and $\gamma\in
\Autz(\HYPG(\uple{\chi},\uple{\rho}))$ so that
$$
\gamma^*\HYPG(\uple{\chi},\uple{\rho})\simeq
\HYPG(\uple{\chi},\uple{\rho})\otimes\sheaf{L}
$$
for some rank $1$ sheaf $\sheaf{L}$. The right-hand side is ramified
at $\{0,1,\infty\}$ (because $n\geq 2$ and the description of local
monodromy from~\cite[Th. 8.4.2 (8)]{katz-esde} shows that the
ramification of the hypergeometric sheaf cannot be eliminated by
tensoring with a character), and hence $\gamma$ must permute the
points $0$, $1$, $\infty$. This shows that $\gamma\in\Gamma$.
\par
(5) Arguing as in (4) with an isomorphism
$$
\gamma^*\HYPG(\uple{\chi},\uple{\rho})\simeq
\dual(\HYPG(\uple{\chi},\uple{\rho}))\otimes\sheaf{L}
$$
we see that $\Autt(\HYPG(\uple{\chi},\uple{\rho}))\subset
\Gamma$. Then the statement is just the conclusion of
Lemma~\ref{lm-special-coset} in this special case.
\end{proof}

%%Example:
%% (chi_1,chi_1,chi_1) and (), with Lambda=chi_1^2.

\begin{remark}\label{rm-mellin}
  (1) A different approach, which is natural from the analytic point
  of view, would be to study such questions by means, for instance, of
  the sums
$$
S_E=\sum_{t\in E^{\times}} \hypg(\uple{\chi},\uple{\rho},t;E)
\overline{\hypg(\uple{\chi},\uple{\rho},at;E)}
$$
for finite extensions $E/k$, where
$\hypg(\uple{\chi},\uple{\rho},t;E)$ denotes the natural extension of
hypergeometric sums to $E$, using the additive character $\psi_E$
defined by composing $x\mapsto e(x/p)$ with the trace from $E$ to
$\Fp$.  The Riemann Hypothesis implies that if~(\ref{eq-hyp-iso})
holds, then
$$
\liminf_{|E|\ra +\infty}\frac{|S_E|}{|E|}>0.
$$
\par
Using the Plancherel formula and the fact that the Mellin transform of
a hypergeometric sum is a product of Gauss sums
(see~\cite[8.2.8]{katz-esde}), one gets for $a=1$ the formula
\begin{equation}\label{eq-sp-mellin}
S_E=\frac{1}{|E^{\times}|} \sum_{\Lambda}\prod_{i}
g(\psi_E,\Lambda\chi_{1,i}) \prod_{i}
g(\psi_E,\Lambda\overline{\chi_{2,i}}) \prod_{j}
g(\psi_E,\Lambda\rho_{1,j}) \prod_{j}
g(\psi_E,\Lambda\overline{\rho_{2,j}})
\end{equation}
where $\Lambda$ runs over multiplicative characters of $E^{\times}$
and
$$
g(\psi,\chi)=\sum_{x}{\chi(x)\psi(x)}
$$
denotes the Gauss sums. But one can get the fact that
$$
\lim \frac{S_E}{|E|}=0
$$
unless $\uple{\chi}_1\sim\uple{\chi}_2$ and $\uple{\rho}_1\sim
\uple{\rho}_2$, using Katz's simultaneous equidistribution theorem for
angles of Gauss sums (see~\cite[Th. 9.5]{katz-gkm}
or~\cite[Cor. 20.2]{katz-mellin}).  The case of $a\not=1$ is however
not as easy with this approach.
\par
It is however very interesting to note how the
expression~(\ref{eq-sp-mellin}) for $S_E$ is a multiplicative analogue
of our typical ``sums of products'', the sum being indexed by
multiplicative characters, and involving products of functions defined
on the set of multiplicative characters. From this point of view, the
proof of equidistribution of Gauss sums in~\cite{katz-mellin} is the
most natural, as it relies on the analogue of the geometric monodromy
group discovered by Katz in this context (using Tannakian formalism
among other things), although the relevant group is a direct product
of copies of $\Gg_m$ (see~\cite[Lemma 20.1]{katz-mellin}), which we
never handle in this paper.
\par
It would be possible (and of some interest, although we do not have
concrete applications to analytic number theory in mind at the moment)
to extend the theory of ``sums of products'' to deal with Mellin
transforms of trace functions instead of trace functions, with Katz's
symmetry group replacing the geometric monodromy group.
\par
(2) It may be that a hypergeometric sheaf satisfies
$$
\dual(\HYPG(\uple{\chi},\uple{\rho}))\simeq
\HYPG(\uple{\chi},\uple{\rho})\otimes\sheaf{L},\quad\text{or}\quad
\HYPG(\uple{\chi},\uple{\rho}))\simeq
\HYPG(\uple{\chi},\uple{\rho})\otimes\sheaf{L},
$$
for some rank $1$ sheaf $\sheaf{L}$; this is however a different
question than the one addressed for applications to sums of products.
For instance, we have geometric isomorphisms
$$
\HYPG((1,\chi_1),(\chi_2,\chi_3))\simeq \sheaf{L}_{\chi_4(X-1)}\otimes
\HYPG((1,\chi_1),(\overline{\chi}_2\chi_1,\overline{\chi}_3\chi_1))
$$
for multiplicative characters $\chi_1$, $\chi_2$, $\chi_3$ with
$\chi_1\notin \{\chi_2,\chi_3\}$ and
$$
\chi_4=\chi_2\chi_3\overline{\chi}_1
$$
(analogues of the Euler identity~\cite[9.131.1 (3)]{gradstheyn} for the
$_{2}F_{1}$-hypergeometric function). %%TODO: check for reference?
If $\chi_1$ is of order $2$, $\chi_2$ is of order $4$ such that
$\chi_2^2=\chi_1$ and $\chi_3=\chi_1\chi_2$, then we obtain
$$
\HYPG((1,\chi_1),(\chi_2,\chi_3))\simeq \sheaf{L}_{\chi_1(X-1)}\otimes
\HYPG((1,\chi_1),(\chi_2,\chi_3)),
$$
since $\chi_4=\chi_1$ in that case.
\par
(3) At least some of the restrictions on $(n,m)$ in
Proposition~\ref{pr-aut-hyper} are necessary. For instance, for
$(n,m)=(2,2)$, we have geometric isomorphisms
$$
\gamma^*\HYPG((1,\chi_1),(\chi_2,\overline{\chi}_3\chi_1))
\simeq \sheaf{L}_{\chi_2(X-1)}\otimes
\HYPG((1,\chi_1),(\chi_2,\chi_3))
$$
where
$$
\gamma=\begin{pmatrix}1&0\\1&-1
\end{pmatrix},\quad\text{i.e.}\quad
\gamma\cdot x=\frac{x}{x-1}
$$
(analogue of~\cite[9.131.1 (1)]{gradstheyn}). If $\chi_3$ satisfies
$\chi_3^2=\chi_1$, and $\chi_2$ is non-trivial, we deduce that
$\gamma\in\Autz(\HYPG((1,\chi_1),(\chi_2,\chi_3))$.
\par
(4) One can be more precise concerning the case $m=n$, for any given
concrete choice of characters, but we did not attempt to obtain a full
classification. For instance, concerning
$\Autt(\HYPG(\uple{\chi},\uple{\rho}))$ in that case, the reader can
easily classify the possibilities of subgroups $H\subset \Gamma$ and
$\xi\in N_{\Gamma}(H)$ such that $\xi^2\in H$. Thus any concrete case
can most likely be analyzed in order to determine exactly
$\Autz(\HYPG(\uple{\chi},\uple{\rho}))$ and
$\Autt(\HYPG(\uple{\chi},\uple{\rho}))$.
\end{remark}

In view of these results, one can feel confident that sums of products
of hypergeometric sums can be handled using the results of this paper,
at least in many cases. The trickiest case would be when
$G^0=\Ort_{r}$ with $r$ even (in view of Remark~\ref{rm-so} (2)),
which does occur (e.g., if $n-m\geq 2$ is even, $n$ is even, the
tuples $\uple{\chi}$ and $\uple{\rho}$ are stable under inversion, and
$\prod\chi_i$ is non-trivial of order $2$, see~\cite[Th. 8.8.1, Lemma
8.11.6]{katz-esde}).

\subsection{Fourier transforms of multiplicative characters}

Many examples of sheaves with suitable monodromy groups are discussed
in~\cite[7.6--7.14]{katz-esde}, arising from Fourier transforms of
other (rather simple) sheaves. We discuss one illustrative case,
encouraging the reader to look at Katz's results if she encounters
similar-looking constructions.
\par
We consider a polynomial $g\in\Fp[X]$ and a non-trivial multiplicative
character $\chi$ modulo $p$. We assume that no root of $g$ is of order
divisible by the order of $\chi$. We then form the sheaf
$$
\sheaf{F}_{\chi,g}=\ft_{\psi}(\sheaf{L}_{\chi(g)})
$$
i.e., the Fourier transform of the Kummer sheaf with trace function
$\chi(g(x))$, where $\psi$ is the additive character $e(\cdot/p)$. The
trace function of $\sheaf{F}_{\chi,g}$ is
$$
K_{\chi,g}(x)=-\frac{1}{\sqrt{p}}
\sum_{y\in\Fp}\chi(g(y))\psi(xy).
$$

\begin{proposition}
  With notation as above, let $r$ be the number of distinct roots of
  $g$ in $\bFp$. Assume that $r\geq 2$ and $p>2r+1$. Assume
  furthermore that the only solutions of the equations
\begin{equation}\label{eq-root-energy}
x_1-x_2=x_3-x_4
\end{equation}
where $(x_1,\ldots,x_4)$ range over the roots of $g$ in $\bFp$ are
given by $x_3=x_1$, $x_4=x_2$ or $x_2=x_1$ and $x_3=x_4$. Then
$\sheaf{F}_{\chi,g}$ is a middle-extension sheaf of weight $0$, of
rank $r$, lisse on $\Gg_m$, and with geometric monodromy group
containing $\SL_r$. Furthermore, we have
$$
\Autz(\sheaf{F}_{g,\chi})\simeq \{a\in\bFp\,\mid\, g(a
X)=cg(X-\alpha)\text{ for some } c\in\Fp^{\times}, \ \alpha\in \Fp\},
$$
and
\begin{equation}\label{eq-dual-chi}
\Autt(\sheaf{F}_{\chi,g})=\emptyset
  % \{\gamma\in\PGL_2(\bFp)\,\mid\, \gamma^*\sheaf{F}_{\chi,g}\simeq
  % \dual(\sheaf{F}_{\chi,g})\otimes\sheaf{L}\text{ for some rank $1$
  %   sheaf }
  % \sheaf{L}\}=1.
\end{equation}
if $r\geq 3$.
\end{proposition}

Note that if $r=2$, the sheaf is of $\Sp_2$-type (since $\Sp_2=\SL_2$)
to that $\Autt(\sheaf{F}_{\chi,g})$ is not relevant in that case. 

\begin{proof}
  The sheaf $\sheaf{L}_{\chi(g)}$ is an irreducible tame
  pseudoreflection sheaf in the sense
  of~\cite[7.9.1--7.9.3]{katz-esde}, ramified at the zeros of $g$
  (because of our assumption on their order) hence the fact that
  $\sheaf{F}_{\chi,g}$ is lisse on $\Gg_m$ and of rank $r$ follows
  from~\cite[Th. 7.9.4]{katz-esde}.  It is a middle extension,
  pointwise of weight $0$, by the general theory of the Fourier
  transform. Moreover, by~\cite[Th. 7.9.6]{katz-esde}, the geometric
  monodromy group contains $\SL_r$ because of the assumptions on the
  roots.
\par
We next compute the projective automorphism group. We first note that
because $r\geq 2$, there is at least one non-zero root, and hence
$\sheaf{F}_{\chi,g}$ is wildly ramified at $\infty$ by~\cite[Th. 7.9.4
(2)]{katz-esde}. On the other hand, it is ramified, but tame, at $0$
by~\cite[7.4.5 (2)]{katz-esde}.
\par
Now assume $\gamma\in \Autz(\sheaf{F}_{\chi,g})$, and that
$\sheaf{L}$ is a rank $1$ sheaf such that
$$
\gamma^*\sheaf{F}_{\chi,g}\simeq \sheaf{F}_{\chi,g}\otimes\sheaf{L}.
$$
\par
We first claim that $\gamma$ is diagonal or anti-diagonal. Indeed, if
$\gamma^{-1}(0)\notin\{0,\infty\}$, the sheaf $\sheaf{L}$ must be
ramified at $\gamma^{-1}(0)$ for the tensor product to to be ramified
there, as $\gamma^*\sheaf{F}_{\chi,g}$ is. But then the
pseudoreflection monodromy means that the inertia invariants have
codimension $1$ on the left, and $r$ on the right (since the stalk of
$\sheaf{L}$ at $\gamma^{-1}(0)$ must vanish). Since $r\geq 2$, this is
not possible. Similarly, $\gamma^{-1}(\infty)\in\{0,\infty\}$, proving
the claim.
\par
Next, we can see that in fact $\gamma$ must be diagonal. Indeed,
otherwise $\sheaf{F}_{\chi,g}\otimes\sheaf{L}$ would be tame at
$\infty$, but this is not possible. Indeed, as a representation of the
wild inertia group at $\infty$, this tensor product is isomorphic to
the representation
$$
\bigoplus_{x}\sheaf{L}_{\psi(xX)}\otimes\sheaf{L}
$$
where the sum ranges over zeros of $g$ in $\bFp$, by~\cite[Th. 7.9.4
(2)]{katz-esde}. There are at least two summands since $r\geq 2$, and
if one is tame, say that of $x$, then for any other zero $x'\not=x$,
we have
$$
\sheaf{L}_{\psi(x'X)}\otimes\sheaf{L}\simeq
\Bigl(\sheaf{L}_{\psi(xX)}\otimes\sheaf{L}\Bigr)
\otimes \sheaf{L}_{\psi((x'-x)X)}
$$
which is \emph{not} tame as tensor product of a tamely ramified and a
wildly ramified character. Thus the direct sum contains at least one
wildly ramified summand.
\par
We are thus left with the case where $\gamma\cdot x=ax$ for some
$a\in\bFp^{\times}$. Now, assume we have
$$
[\times a]^*\sheaf{F}_{\chi,g}\simeq \sheaf{F}_{\chi,g}\otimes\sheaf{L}.
$$
\par
If $\sheaf{L}$ were ramified at some $x\in\Gg_m(\bFp)$, the right-hand
side would also be (since $\sheaf{F}_{\chi,g}$ is lisse on $\Gg_m$),
but the left-hand side is not. Hence $\sheaf{L}$ is lisse on $\Gg_m$.
\par
Furthermore, $\sheaf{L}$ is at most tamely ramified at $0$, since
$[\times a]^*\sheaf{F}_{\chi,g}$ is.  Let $\Lambda$ be the tame
character of the inertia group at $0$ which corresponds to
$\sheaf{L}$.  By~\cite[Cor. 7.4.6(1)]{katz-esde}, the sheaves
$\sheaf{F}_{\chi,g}$ and $[\times a]^*\sheaf{F}_{\chi,g}$ both have
pseudoreflection monodromy at $0$ with inertia group at $0$ acting on
the inertial invariants by the character
$\sheaf{L}_{\overline{\chi}(X)}$. Thus our assumed geometric
isomorphisms leads to
$$
\sheaf{L}_{\overline{\chi}(X)}\simeq \sheaf{L}_{\overline{\chi}(X)}
\otimes\sheaf{L}_{\Lambda},
$$
and therefore to $\Lambda=1$. Hence $\sheaf{L}$ is unramified at $0$.
\par
Looking again at infinity, we find an isomorphism 
$$
\bigoplus_{x}\sheaf{L}_{\psi(axX)}\simeq
\bigoplus_{x}\sheaf{L}_{\psi(xX)}\otimes\sheaf{L}
$$
of representations of the wild inertia group. Picking one root $x_i$,
we deduce that $\sheaf{L}$ is isomorphic to $\sheaf{L}_{\psi(\alpha
  X)}$ for some $\alpha$, as a representation of the wild inertia
group at infinity. Hence we have a geometric isomorphism
$$
\sheaf{L}\simeq \sheaf{L}_{\psi(\alpha X)}
$$
since $\sheaf{L}\otimes\sheaf{L}_{\psi(-\alpha X)}$ is of rank $1$,
lisse on $\Aa^1$ and tame on $\Pp^1$, hence geometrically trivial.
\par
Finally, using the inverse Fourier transform, we see that
$$
[\times a]^*\sheaf{F}_{\chi,g}\simeq
\sheaf{F}_{\chi,g}\otimes\sheaf{L}_{\psi(\alpha X)}
$$
is equivalent to
$$
\sheaf{L}_{\chi(g(X/a))}\simeq \sheaf{L}_{\chi(g(X-\alpha))},
$$
which is equivalent (by comparing degrees and using the classification
of Kummer sheaves) to 
$$
g(X/a)=cg(X-\alpha)
$$
for some constants $c\in\Fp^{\times}$ and $\alpha\in\Fp$. This gives
the stated result concerning $\Autz(\sheaf{F}_{\chi,g})$.
\par
For the last statement, assume that $r\geq 3$ and that $\gamma\in
\Autt(\sheaf{F}_{\chi,g})$, i.e., that we have
$$
\gamma^*\sheaf{F}_{\chi,g}\simeq
\dual(\sheaf{F}_{\chi,g})\otimes\sheaf{L}
$$
for some rank $1$ sheaf $\sheaf{L}$. Exactly as before, we see first
that $\gamma$ is diagonal or anti-diagonal, and then that it is
diagonal, by considering ramification. Then we see that $\sheaf{L}$ is
tame at $0$, and in fact the tame character by which it acts at $0$ is
$\bar{\chi}^2$.
\par
Next, as representations of the wild inertia group at $\infty$, we
obtain
$$
\bigoplus_x \sheaf{L}_{\psi(xX)}\simeq \bigoplus_x
\sheaf{L}_{\psi(-xX)}\otimes\sheaf{L}.
$$
\par
We deduce that $\sheaf{L}$ must be of the form $\sheaf{L}_{\psi(\alpha
  X)}$ for some $\alpha$, as a representation of the wild inertia
group at infinity.  This means that if $x$ is a root of $g$, then so
is $\alpha-x$. But since there are at least three distinct roots of
$g$, we can fix some root $x$ of $g$ and find another root $y\notin
\{x,\alpha-x\}$.  Then the equation
$$
x-(\alpha-y)=y-(\alpha-x)
$$
contradicts our assumption on the roots of~(\ref{eq-root-energy}).
\end{proof}

% \begin{example}
% For instance, taking $g=X^p-X$ shows that it is possible to 
% \end{example}

\section{Sums of products with fractional linear transformations}\label{sec-proofs}

We can now quickly prove the results stated in Section~\ref{sec-intro}
using the framework established previously.

% \begin{lemma}
%   \emph{(1)} For any semisimple connected linear algebraic group $G$
%   over an algebraically closed field $k$ of characteristic zero, and
%   any faithful irreducible, self-dual, representation $\rho$ of $G$,
%   there exists $r\geq 0$ such that, for $k\geq 0$, the trivial
%   representation is a subrepresentation of $\rho^{\otimes k}$ if and
%   only if $r$ divides $k$. We have $r\leq 1$ if and only if $G=1$, and
%   $r=2$ otherwise.
% \par
% \emph{(2)} For any linear algebraic group $G$ over an algebraically
% closed field $k$ of characteristic zero, and any faithful irreducible,
% non self-dual, representation $\rho$ of $G$, there exists $r\geq 0$
% such that, for $k_1$, $k_2\geq 0$, the trivial representation is a
% subrepresentation of $\rho^{\otimes k_1}\otimes\dual(\rho_2)^{\otimes
%   k_2}$ if and only if $r$ divides $k_1-k_2$. We have $r\leq 1$ if and
% only if $G=1$, and $r=n$ if $G=\SL_n$ for $n\geq 3$.
% \end{lemma}

% \begin{proof}
% (1) The set
% $$
% I=\{k\geq 0\,\mid\, \mathbf{1}\text{ is a subrepresentation of }
% \rho^{\otimes k}\}
% $$
% contains $0$ and is stable under addition. It contains some positive
% element (because $G\injecte \SL_k$ for some $k\geq 3$, and 
% \end{proof}

\begin{proof}[Proof of Theorem~\ref{th-main1}]
  First, we denote by $U$ the common open set in $\Aa^1$ where all
  $\gamma\in\uple{\gamma}^*$ are defined.
\par
We begin with the easier $\Sp$-type case.  Let $\uple{\gamma}^*$ be
the tuple of distinct elements of $\uple{\gamma}$, and $n_{\gamma}$
the multiplicity of any such element in $\uple{\gamma}$. Let $U$ be
the common open set in $\Aa^1$ where all $\gamma\in\uple{\gamma}^*$
are defined. Arguing as in Example~\ref{ex-old} (2), we see that the
tuple
$\uple{\sheaf{F}}=(\gamma^*\sheaf{F})_{\gamma\in\uple{\gamma}^*}$ is
strictly $U$-generous, simply because $\sheaf{F}$ is bountiful of
$\Sp_r$ type.
\par
By the birational invariance of $H^2_c$, we have
$$
H^2_c(\Aa^1\times\bFp, \bigotimes_{1\leq i\leq k}
\gamma_i^*\sheaf{F}\otimes \sheaf{L}_{\psi(hX)})= H^2_c(U\times\bFp,
\bigotimes_{1\leq i\leq k} \gamma_i^*\sheaf{F}\otimes
\sheaf{L}_{\psi(hX)}).
$$
\par
Thus, by Theorems~\ref{th-diag} and~\ref{th-diag-2}, we see that if
$$
H^2_c(\Aa^1\times\bFp, \bigotimes_{1\leq i\leq k}
\gamma_i^*\sheaf{F}\otimes \sheaf{L}_{\psi(hX)})\not=0,
$$
there must exist some geometric isomorphism
$$
 \sheaf{L}_{\psi(hX)}\simeq
\bigotimes_{\gamma\in\uple{\gamma}^*}\Lambda_{\gamma}(\gamma^*\sheaf{F})
$$
where $\Lambda_{\gamma}$ are irreducible representations of the
geometric monodromy group $G=\Sp_r$ of $\sheaf{F}$ such that
$\Lambda_{\gamma}$ is a subrepresentation of $\std^{\otimes
  n_{\gamma}}$. Just for dimension reasons, each $\Lambda_{\gamma}$
must be a one-dimensional character. But
Definition~\ref{def-gen-sheaf} implies in particular that $G$ has no
non-trivial character, so that $\Lambda_{\gamma}=1$, which implies
that $\sheaf{L}_{\psi(hX)}$ must be geometrically trivial, i.e., that
$h=0$.
\par
This already proves the first part of Theorem~\ref{th-main1} when
$h\not=0$. Now assume $h=0$. Then the condition that the trivial
representation be a subrepresentation of $\std^{\otimes n_{\gamma}}$
holds if and only if $n_{\gamma}$ is even, and thus the $H^2_c$ space
does \emph{not} vanish if and only if all multiplicities $n_{\gamma}$
are even, which means if and only if $\uple{\gamma}$ is \emph{not}
normal.
\par
We now come to the $\SL_r$-type case. If $\sheaf{F}$ has a special
involution $\xi$, let $\sheaf{L}$ be a rank $1$ sheaf such that
\begin{equation}\label{eq-xi-special}
\xi^*\sheaf{F}\simeq \dual(\sheaf{F})\otimes\sheaf{L},
\end{equation}
and we note that (as a character of the fundamental group of
$U\times\bFp$) the sheaf $\sheaf{L}$ has order dividing $r$ (by taking
the determinant on both sides).
\par
For convenience, we let $\xi=1$ and $\sheaf{L}=\bQl$, if there is no
special involution.
\par
Let $\uple{\gamma}^*$ be a tuple of representatives of the elements of
$\uple{\gamma}$ for the equivalence relation
$$
\gamma_i\sim \gamma_j\text{ if and only if } (\gamma_i=\gamma_j\text{
  or } \gamma_i=\xi\gamma_j)
$$
(which is indeed an equivalence relation because $\xi^2=1$).
\par
Then, arguing as in Example~\ref{ex-old} (3), we see that the tuple
$\uple{\sheaf{F}}=(\gamma^*\sheaf{F})_{\gamma\in\uple{\gamma}^*}$ is
strictly $U$-generous, because $\sheaf{F}$ is bountiful of
$\SL_r$-type and because
$$
\gamma_i^*\sheaf{F}\simeq \dual(\gamma_j^*\sheaf{F})\otimes\sheaf{L}',
$$
for some rank $1$ sheaf $\sheaf{L}'$, implies that
$$
\gamma_i\gamma_j^{-1}\in\Autt(\sheaf{F}),
$$
and thus either does not occur (if $\sheaf{F}$ has no special
involution) or happens only if $\gamma_i=\xi\gamma_j$, so that
$\gamma_i\sim\gamma_j$, which is excluded for distinct components of
$\uple{\gamma}^*$.
\par
For $\gamma\in\uple{\gamma}^*$, we denote
\begin{align*}
n^{1}_{\gamma}&=
|\{i\,\mid\, \gamma_i=\gamma\text{ and }
\sigma_i=1\}|+
|\{i\,\mid\, \gamma_i=\xi\gamma\text{ and }
\sigma_i=c\}|,
\\
n^{c}_{\gamma}&=
|\{i\,\mid\, \gamma_i=\gamma\text{ and }
\sigma_i=c\}|+
|\{i\,\mid\, \gamma_i=\xi\gamma\text{ and }
\sigma_i=1\}|,
\end{align*}
so that, by bringing together equivalent $\gamma_i$'s, we obtain a
geometric isomorphism
\begin{equation}\label{eq-reduce}
\bigotimes_{1\leq i\leq k}\gamma_i^*(\sheaf{F}^{\sigma_i}) \simeq
\bigotimes_{\gamma\in\uple{\gamma}^*} (\gamma^*\sheaf{F})^{\otimes
  n_{\gamma}^1} \otimes \dual(\gamma^*\sheaf{F})^{\otimes
  n_{\gamma}^c}\otimes\sheaf{L}_0
\end{equation}
for some rank $1$ sheaf $\sheaf{L}_0$, which is a tensor product of
sheaves of the form $\gamma^*\sheaf{L}$ or
$\gamma^*(\dual{\sheaf{L}})$. In particular, $\sheaf{L}_0$ has order
dividing $r$ since $\sheaf{L}$ does.
\par
We now get from Theorem~\ref{th-diag-2} that if
\begin{multline*}
  H^2_c(\Aa^1\times\bFp, \bigotimes_{1\leq i\leq k}
  \gamma_i^*(\sheaf{F}^{\sigma})\otimes \sheaf{L}_{\psi(hX)})\\=
  H^2_c(\Aa^1\times\bFp,\bigotimes_{\gamma\in\uple{\gamma}^*}
  (\gamma^*\sheaf{F})^{\otimes n_{\gamma}^1} \otimes
  \dual(\gamma^*\sheaf{F})^{\otimes n_{\gamma}^c}\otimes
  (\sheaf{L}_0\otimes \sheaf{L}_{\psi(hX)})) \not=0,
\end{multline*}
then 
$$
\sheaf{L}_0\otimes \sheaf{L}_{\psi(hX)}\simeq
\bigotimes_{\gamma\in\uple{\gamma}^*}\Lambda_{\gamma}(\gamma^*\sheaf{F})
$$
where $\Lambda_{\gamma}$ is an irreducible representation of $\SL_r$
which is a subrepresentation of the tensor product $\std^{\otimes
  n_{\gamma}^1}\otimes\dual(\std)^{\otimes n^{c}_{\gamma}}$. Since
$\SL_r$ has no $1$-dimensional characters, this shows that this
condition cannot occur unless $\Lambda_{\gamma}$ is trivial for all
$\gamma$, which implies then that
\begin{equation}\label{eq-l-psi}
  \sheaf{L}_0\otimes\sheaf{L}_{\psi(hX)}\simeq \bQl
\end{equation}
is trivial.
\par
If $\sheaf{F}$ has no special involution, this immediately implies
that $h=0$.  If $\sheaf{F}$ has a special involution, on the other
hand, we recall that $\sheaf{L}_0$ has order $r$, while
$\sheaf{L}_{\psi(hX)}$ has order $p$ if
$h\not=0$. Hence~(\ref{eq-l-psi}) is impossible if $p>r$ and
$h\not=0$, and moreover, in that case we also get
from~(\ref{eq-l-psi}) that $\sheaf{L}_0$ must be trivial.
\par
Thus, in all cases of Theorem~\ref{th-main1}, we reduce to
understanding the case $h=0$. Since $\Lambda_{\gamma}$ is trivial, we
have also the condition that the trivial representation is a
subrepresentation of the tensor product
$$
(\gamma^*\sheaf{F})^{\otimes n_{\gamma}^1} \otimes
\dual(\gamma^*\sheaf{F})^{\otimes n_{\gamma}^c}
$$
for all $\gamma$ in $\uple{\gamma}^*$.
\par
But the trivial representation of $\SL_r$ is a subrepresentation of
$\std^{\otimes n}\otimes\dual(\std)^{\otimes m}$ if and only if $r\mid
n-m$ (see, e.g.,~\cite[Proof of Prop. 4.4]{kr}), and this means that
$H^2_c$ non-zero implies that $r\mid n_{\gamma}^1-n_{\gamma}^c$ for
all $\gamma\in\uple{\gamma}^*$, which means precisely that
$(\uple{\gamma},\uple{\sigma})$ is not $r$-normal (if there is no
special involution) or not $r$-normal with respect to $\xi$ (if there
is one).
\end{proof}

\begin{remark}
We see from the proof that the condition $p>r$ in
Theorem~\ref{th-main1} (when $\sheaf{F}$ has a special involution) can
be relaxed: especially, it is not needed if we have
$$
\xi^*\sheaf{F}\simeq \dual(\sheaf{F})
$$
(i.e. if $\sheaf{L}$ in~(\ref{eq-xi-special}) can be taken to be the
trivial sheaf, since we only used $p>r$ to deduce that $\sheaf{L}_0$
in~(\ref{eq-reduce}) is trivial, which is automatically true in this
case).
\end{remark}

For completeness, we explain the proof of Proposition~\ref{pr-rh}:

\begin{proof}[Proof of Proposition~\ref{pr-rh}]
Let $U\subset \Aa^1$ be the maximal open set where all sheaves
$\sheaf{F}_i$ and $\sheaf{G}$ are lisse. We have
$$
|(\Aa^1-U)(\Fp)|\leq \sum_i\cond(\sheaf{F}_i)+\cond(\sheaf{G}).
$$
\par
Since the sheaves are all mixed of weights $\leq 0$, we have
$$
\Bigl|\sum_{x\in U(\Fp)}
K_1(x)\cdots K_k(x)\overline{M(x)}
-
\sum_{x\in\Fp}
K_1(x)\cdots K_k(x)\overline{M(x)}\Bigr|\leq C_1|(\Aa^1-U)(\Fp)|
$$
where $C_1$ is the product of the ranks of the sheaves.  This means
that it is enough to deal with the sum over $x\in U(\Fp)$.
\par
By the Grothendieck--Lefschetz trace formula we have
$$
\sum_{x\in U(\Fp)} K_1(x)\cdots K_k(x)\overline{M(x)}= -\Tr(\frob\mid
H^1_c(U\times\bFp, \bigotimes_{i}\sheaf{F}_i\otimes\dual(\sheaf{G})))
$$
since the $H^0_c$ and $H^2_c$ terms vanish, by assumption for $H^2_c$
and because we have a tensor product of middle-extension sheaves for
$H^0_c$.
\par
By Deligne's proof of the Riemann Hypothesis~\cite{weilii}, since the
tensor product is of weight $0$, all eigenvalues of Frobenius acting
on the cohomology space have modulus $\leq \sqrt{p}$, and hence
$$
\Bigl|\sum_{x\in U(\Fp)}
K_1(x)\cdots K_k(x)\overline{M(x)}
\Bigr|\leq \dim H^1_c(U\times\bFp,
\bigotimes_{i}\sheaf{F}_i\otimes\dual(\sheaf{G}))
\times \sqrt{p}.
$$
\par
Finally, using the Euler-Poincaré formula, one sees that the dimension
of this space is bounded in terms of the conductors of $\sheaf{F}_i$
and of $\sheaf{G}$, and in terms of $k$.
\end{proof}

As already mentioned, Corollary~\ref{cor-concrete} is an immediate
consequence of Theorem~\ref{th-main1} and Proposition~\ref{pr-rh}.
Corollary~\ref{cor-concrete2} is similar, except that in the argument
of Proposition~\ref{pr-rh}, there is a main term in the trace formula
which is (for the $\Sp$-type case) given by
$$
\Tr(\frob\mid H^2_c(\bigotimes
\gamma_i^*\sheaf{F}\otimes\dual(\sheaf{G}))
).
$$
\par
However, the extra assumption that the geometric monodromy group
coincides with the arithmetic monodromy group means that all
eigenvalues of the Frobenius acting on $H^2_c$ are equal to $p$. Hence
this contribution is equal to
$$
p\dim H^2_c(\bigotimes \gamma_i^*\sheaf{F}\otimes\dual(\sheaf{G})),
$$
and for $\sheaf{G}$ given (as in the proof of Theorem~\ref{th-main1})
by
$$
\sheaf{G}=\bigotimes_{\gamma\in\uple{\gamma}^*}
\Lambda_{\gamma}(\gamma^*\sheaf{F})
$$
with $\Lambda_{\gamma}$ an irreducible representation of $G$ which is
a subrepresentation of $\std^{\otimes n_{\gamma}}$ (as it must be to
have non-zero $H^2_c$), we have
$$
\dim H^2_c(\bigotimes \gamma_i^*\sheaf{F}\otimes\dual(\sheaf{G}))=
\prod_{\gamma\in\uple{\gamma}^*}
\mathrm{mult}_{\Lambda_{\gamma}}(\std^{\otimes n_{\gamma}})
$$
where each multiplicity is at most $k$, and is equal to $1$ if
$n_{\gamma}=1$. The result follows immediately. The case of
$\SL_r$-type is similar and left to the reader; the extra condition
that $\xi^*\sheaf{F}\simeq \dual(\sheaf{F})$ (without a twist by a
non-trivial rank $1$ sheaf) allows us to deduce~(\ref{eq-reduce}) with
$\sheaf{L}_0$ trivial, from which the non-vanishing of $H^2_c$ follows
when $(\uple{\gamma},\uple{\sigma})$ is not $r$-normal with respect to
the special involution. (We already observed that under this condition
we do not need to assume $p>r$ in Theorem~\ref{th-main1}).

\section{Applications}\label{sec-applications}

We present here some applications of the general case developed in
Section~\ref{sec-general}, going beyond the results of the
introduction and of the previous section. The first recovers an
estimate of Katz used by Fouvry and Iwaniec in their study of the
divisor function in arithmetic progressions~\cite{fouvry-iwaniec-div},
the second discusses briefly the sums of Bombieri and
Bourgain~\cite{bombieri-bourgain}, while the last only is a new
result, which is related to the context of~\cite{FGKM, kr}. We also
recall the occurence of this type of situations in the work of Fouvry,
Michel, Rivat and S\'ark\"ozy~\cite[Lemma 2.1]{FMRS}, although we will
not review it.
% , as well in work of Bombieri and Bourgain~\cite[Lemma
% 33]{bombieri-bourgain} (see also~\cite{katz-bb}).

\subsection{The Fouvry-Iwaniec sum} In~\cite{fouvry-iwaniec-div}, for
primes $p$ and $(\alpha,\beta)\in\Fpt^2$, the exponential sum
$$
S(\alpha,\beta;p)= \sums_{t}
\hypk_2(\alpha(t-1)^2)\hypk_2((t-1)(\alpha t-\beta))
\hypk_2(\beta(t^{-1}-1)^2)\hypk_2((t^{-1}-1)(\beta t^{-1}-\alpha))
$$
arises, where the sum is over $t\in\Fpt-\{1,\beta/\alpha\}$, and we
abbreviate $\hypk_2(x)=\hypk_2(x;p)$. This is not of the type of
Section~\ref{sec-intro}, since the arguments of the Kloosterman sums
are not simply of the form $\gamma_i\cdot t$. However, it fits the
general framework of Section~\ref{sec-general} with the $4$-tuple
$$
\uple{\sheaf{F}}=(f_i^*\HYPK_2)_{1\leq i\leq 4},
$$
where
\begin{gather*}
f_1=\alpha(X-1)^2,\quad\quad f_2=(X-1)(\alpha X-\beta)\\
f_3=\beta (X^{-1}-1)^2,\quad\quad
f_4=(X^{-1}-1)(\beta X^{-1}-1).
\end{gather*}
\par
Let $U=\Gg_m-\{1,\beta/\alpha\}$. We claim that this $4$-tuple is
$U$-generous if $\alpha\not=\beta$ (which is certainly a necessary
condition, since otherwise $f_1=f_2$). Indeed, since the geometric
monodromy group of each $f_i^*\sheaf{F}$ is $\SL_2=\Sp_2$ (because the
geometric monodromy group of $\HYPK_2$ is $\SL_2$, and $\SL_2$ has no
finite index algebraic subgroup), we need to check that there is no
geometric isomorphism
$$
f_i^*\HYPK_2\simeq f_j^*\HYPK_2\otimes\sheaf{L}
$$
for $i\not=j$ and a rank $1$ sheaf $\sheaf{L}$. But taking the dual
and then tensoring, such an isomorphism implies
$$
f_i^*\End(\HYPK_2)\simeq f_j^*\End(\HYPK_2),
$$
on the open set $V=f_i^{-1}(\Gg_m)$ where the left-hand side of the
original isomorphism (hence also the right-hand side) is lisse. Since
$\End(\HYPK_2)\simeq \bQl\oplus \symk^2(\HYPK_2)$, this implies that
$$
f_i^*\symk^2(\HYPK_2)\simeq f_j^*\symk^2(\HYPK_2),
$$
on $V$.
\par
But since $\symk^2(\HYPK_2)$ is ramified at $0$ and $\infty$, the
ramification loci $S_i$ of the sheaves $f_i^*\symk^2(\HYPK_2)$ are,
respectively
\begin{gather*}
  S_1=\{1,\infty\},\quad\quad S_2=\{1,\beta/\alpha,\infty\},\\
  S_3=\{0,1\},\quad\quad S_4=\{0,1,\beta\},
\end{gather*}
and are therefore distinct, proving the desired property of
$U$-generosity.
\par
Since the sum $S(\alpha,\beta;p)$ concerns the tensor product of
$$
f_1^*\HYPK_2\otimes f_2^*\HYPK_2\otimes f_3^*\HYPK_2\otimes
f_4^*\HYPK_2
$$
with the trivial sheaf, which is a tensor product of the trivial
representations, which is not a subrepresentation of $\std$, it
follows therefore that
\begin{multline*}
H^2_c(\Aa^1\times\bFp,f_1^*\HYPK_2\otimes f_2^*\HYPK_2\otimes
f_3^*\HYPK_2\otimes f_4^*\HYPK_2)=\\
H^2_c(U\times\bFp,f_1^*\HYPK_2\otimes f_2^*\HYPK_2\otimes
f_3^*\HYPK_2\otimes f_4^*\HYPK_2)=0,
\end{multline*}
and hence by Proposition~\ref{pr-rh} that
$$
S(\alpha,\beta;p)\ll p^{1/2}
$$
for all primes $p$ and $\alpha\not=\beta$ in $\Fpt$, where the implied
constant is absolute. In the Appendix to~\cite{fouvry-iwaniec-div},
Katz gives a precise estimate of the implied constant.

\subsection{The Bombieri-Bourgain sums}
The Bombieri-Bourgain sums are defined by
$$
S=\sum_{x\in\Fp}\prod_{1\leq i \leq k}K_i(x+a_i)M(x)
$$
(see~\cite[p. 513]{katz-bb})
where
\begin{align*}
  M(x)&=e\Bigl(\frac{bx+G(x)}{p}\Bigr)\chi(g(x)),\\
  K_i(x)&=-\frac{1}{\sqrt{p}}\sum_{y\in\Fp}
  \chi_i(f_i(y))e\Bigl(\frac{g_i(y)}{p}\Bigr)
  e\Bigl(\frac{xy}{p}\Bigr)
\end{align*}
for some $b\in\Fp$ and $(a_1,\ldots,a_k)\in\Fp^k$, where
\begin{itemize}
\item 
$(\chi,\chi_1,\ldots,\chi_k)$ are non-trivial multiplicative
characters modulo $p$,
\item  $f_i\in \Fp[X]$, $g\in\Fp[X]$ are non-zero polynomials,
\item $g_i\in\Fp[X]$ and $G\in\Fp[X]$ may be zero.
\end{itemize}
\par
This sum is of the type considered in Section~\ref{sec-general}, with
\begin{align*}
\sheaf{F}_i&=[+a_i]^*
\ft_{\psi}(\sheaf{L}_{\psi(g_i)}\otimes\sheaf{L}_{\chi(f_i)}),\\
\sheaf{G}&=\sheaf{L}_{\psi(G+bX)}\otimes\sheaf{L}_{\chi(g)}
\end{align*}
(or rather those $\sheaf{F}_i$ corresponding to the distinct
parameters since this is not assumed to be the case).
\par
Under (different) suitable conditions on these parameters, Bombieri
and Bourgain~\cite[Lemma 33]{bombieri-bourgain} and
Katz~\cite[Th. 1.1]{katz-bb} give estimates for $S$ of the type
$$
S\ll p^{1/2}
$$
where the implied constant depends only on $k$ and the degrees of the
polynomials involved.  Both proofs avoid involving monodromy groups:
Katz uses the ramification property of Fourier transforms to determine
that the relevant tensor product has zero invariants under some
inertia group, while Bombieri and Bourgain use the Riemann Hypothesis
together with some analytic steps, such as mean-square averaging and
Galois invariance of the weights (this illustrates that sometimes an
estimate for a sum of products might be easier to obtain than those
involved in the previous sections).
\par
We show how to recover quickly the desired square-root cancellation in
the case that occurs for the application considered by Bombieri and
Bourgain, by a hybrid of Katz's argument and those of the previous
sections. 
\par
In~\cite{bombieri-bourgain}, the conditions are: $p$ is odd,
$g_i=G=0$, $1\leq \deg(f_i)\leq 2$, $\deg(g)\geq 2$, the $f_i$ and $g$
have only simple roots, and all $\chi_i$ and $\chi$ are equal and are
of order $2$. We then first note that if some $f_i$ has degree $1$,
the resulting Fourier transform
$$
\ft_{\psi}(\sheaf{L}_{{\chi(f_i)}})
$$
is geometrically isomorphic to a tensor product
$$
\sheaf{L}_{\psi(\alpha X)}\otimes\sheaf{L}_{\chi(X)}
$$
(we use here that $\chi=\bar{\chi}$), so that by combining these with
$\sheaf{G}$ we may assume that all $f_i$ are of degree $2$. Note that
$g$ is replaced by $X^kg$, where $k$ is the number of $i$ with
$\deg(f_i)=1$. Since $\chi$ has order $2$, we have either $k$ even and
$$
\sheaf{L}_{\chi(X^kg)}\simeq \sheaf{L}_{\chi(g)},
$$
so that the previous assumptions on $g$ remain valid, or $k$ odd and
$$
\sheaf{L}_{\chi(X^kg)}\simeq \sheaf{L}_{X\chi(g)}\simeq \sheaf{L}_{\chi(\tilde{g})},
$$
where $\tilde{g}=g/X$ if $g(0)=0$, or $\tilde{g}=Xg$ otherwise; in the
first case it may be that $\deg(\tilde{g})=1$, but in that case the
unique zero of $\tilde{g}$ is in $\Gg_m$ since $g$ has simple
roots. In particular, in all cases, we see that $g$ is replaced by a
polynomial with at least one (simple) root in $\Gg_m$.
\par
If all $f_i$ were of degree $1$, we are left with 
$$
\sum_{x}\chi(g(x))\psi(hx),
$$
with $g$ non-constant, which satisfies the desired conditions. We
therefore assume that some $f_i$ are of degree $2$. 
\par
For a polynomial $f_i$ of degree $2$, by completing squares, we see
that the Fourier transform
$$
\ft_{\psi}(\sheaf{L}_{\chi(f_i)})
$$
is geometrically isomorphic to a tensor product of
$\sheaf{L}_{\psi(hX)}$ for some $h$ and of the Fourier transform
corresponding to a polynomial of the form $X^2+c_i$.  We may therefore
assume that all $f_i$ are of this form.
\par
Finally, it is easy to see that
$$
\ft_{\psi}(\sheaf{L}_{\chi(X^2+c_i)})\simeq [x\mapsto
c_ix^2/4]^*\HYPK_2.
$$
\par
In particular, such sheaves are of rank $2$, lisse on $\Gg_m$ and have
geometric monodromy group $G_i=G_i^0=\SL_2$. We therefore obtain a
strictly $\Gg_m$-generous tuple by taking for $\sheaf{F}_i$ the
Fourier transforms corresponding to the $c_i$'s, modulo the
equivalence relation
$c_i\sim c_j$ if and only if 
$$
c_ic_j^{-1}\in \Autz([x\mapsto x^2/4]^*\HYPK_2).
$$
\par
We can now conclude: since $g$ has a simple zero in $\Gg_m$, the sheaf
$\sheaf{G}$ is ramified at at least one point inside $\Gg_m$, and
therefore the irreducible sheaf $\sheaf{G}$ can not be a subsheaf of
the tensor product
$$
\bigotimes_{i} \sheaf{F}_i^{\otimes n_i}
$$
which is lisse on $\Gg_m$. 

\begin{remark}
  Even if $\deg(g)=1$, $g=\alpha X$ and $\alpha\not=0$, we can obtain
  the square-root bounds provided we have at least one sheaf
  $\sheaf{F}_i$: by the results of Section~\ref{sec-general}, the
  condition
$$
H^2_c(\Gg_m\times\bFp,\bigotimes_{i} \sheaf{F}_i^{\otimes n_i}
\otimes\sheaf{G})\not=0
$$
would imply that $\dual(\sheaf{G})$ is geometrically isomorphic to
$$
\bigotimes_i \mathrm{Sym}^{m_i}(\sheaf{F}_i)
$$
for some $m_i\geq 0$. By rank considerations, we have $m_i=0$, and
this implies that $\sheaf{G}$ is geometrically trivial, which is
impossible since $g$ is non-constant.
\end{remark}

% For a polynomial $f_i$ of degree $2$, the Fourier transform
% $$
% \ft_{\psi}(\sheaf{L}_{\chi(f_i)})
% $$
% is the Fourier transform of an irreducible tame pseudo-reflection
% sheaf with two singularities in $\Aa^1$, in the sense
% of~\cite[Def. (7.9.1)]{katz-esde}. By~\cite[Th. 7.9.6]{katz-esde}, the
% rank of $\ft_{\psi}(\sheaf{F}_{\chi(f_i)})$ is two, it is lisse on
% $\Gg_m$, and (using the fact that the zeros of $f_i$ are simple) the
% geometric monodromy group of $\ft_{\psi}(\sheaf{F}_{\chi(f_i)})$
% contains $\SL_2$. However the main point is that $\sheaf{G}$ can not
% be lisse on $\Gg_m$, since $g$ is a non-constant polynomial with
% simple roots, unless $g=aX$ is proportional to $X$. However, that case
% is excluded by the assumptions in~\cite{bombieri-bourgain}.

\subsection{Central limit theorem for $\GL_N$ cusp forms}

The last example is a generalization of the central limit theorems
of~\cite{FGKM} and~\cite{kr} to residue classes in restricted
subsets. Let $N\geq 2$ be an integer. Fix a smooth function $w\geq 0$
on $[0,+\infty[$, compactly supported on $[1,2]$ and non-zero. For a
cusp form $f$ on $\GL_N$ over $\Qq$, with level $1$, for a prime $p$
and a residue class $a\in\Fpt$, and $X\geq 2$, we denote
$$
E_f(X;p,a)=\frac{1}{(X/p)^{1/2}} \Bigl(\sum_{n\equiv
  a\mods{p}}a_f(n)w(n/X)-\frac{1}{p-1} \sum_{n\geq
  1}a_f(n)w(n/X)\Bigr),
$$
where $a_f(n)$ is the $n$-th Hecke eigenvalue of $f$. Taking
$$
X=p^{N}/\Phi(p),
$$
where $\Phi\geq 1$ is an increasing function such that $\Phi(x)\ll
x^{\eps}$ for all $\eps>0$, it was shown in~\cite{FGKM} (for $N=2$ and
$f$ holomorphic) and in~\cite{kr} (for all other cases) that the
random variables
$$
a\mapsto E_f(X;p,a)
$$
(defined on $\Fpt$ with the uniform measure) converge in law to a
Gaussian, either real (if $f$ is self-dual) or complex (if $f$ is not
self-dual).  Moreover, Lester and Yesha~\cite[Th. 1.2]{lester-yesha}
have shown that if $N=2$, one can replace the smooth weight $w(n/X)$
in the definition of $E_f(X;p,a)$ by the characteristic function of
the interval $[1,X]$.
\par
A natural question (suggested for instance by J-M. Deshouillers) is
whether this central limit theorem persists if $a$ is restricted to a
suitable subset $A_p\subset \Fpt$ (with its own uniform measure). We
explain here that this is indeed the case when $A_p$ has some
algebraic structure.

\begin{theorem}\label{th-clt}
  With notation as above, assume that $A_p$ is:
\par
\emph{(1)} Either a proper generalized arithmetic progression of
dimension $d\geq 1$ with
$$
\limsup \frac{|A_p|}{\sqrt{p}(\log p)^d}=+\infty,
$$
for instance an interval of length $\geq p^{1/2+\delta}$ for some
fixed $\delta>0$;
\par
\emph{(2)} Or the image $g(\Fp)\cap \Fpt$ for a fixed non-constant
polynomial $g\in\Zz[T]$;
\par
Then the random variables restricted to $A_p$ given by
$$
\begin{cases}
A_p\lra \Cc\\
a\mapsto E_f(X;p,a)
\end{cases}
$$
with the uniform probability measure on $A_p$ converge as $p\ra
+\infty$ to the same Gaussian limit as the random variables defined on
all of $\Fpt$.
\end{theorem}

We prove this by first writing the characteristic function of $A_p$ as
a ``short'' linear combination of trace functions, precisely either by
Fourier transform
\begin{equation}\label{eq-ft}
\mathbf{1}_{A_p}(x)=\sum_{h\in\Fp}\alpha_p(h)e\Bigl(\frac{hx}{p}\Bigr)
\end{equation}
with
$$
K_0(x)=1,\quad\quad \alpha_p(0)=\frac{|A_p|}{p}
$$
and
$$
\sum_{h\not=0}|\alpha_p(h)|\ll (\log p)^d
$$
in the first case (this bound is classical for $d=1$, and the case
$d\geq 2$ was proved by Shao~\cite{shao}), or by decomposition in
Artin-like trace functions
\begin{equation}\label{eq-artin}
\mathbf{1}_{A_p}(x)=\sum_{i\in I}\alpha_{i,p}K_i(x)
\end{equation}
in the second case, where $I$ is a finite set depending only on the
polynomial $g$, $0\in I$ with
$$
\alpha_{0,p}=\frac{|A_p|}{p}+O(p^{-1/2}),
$$
and
$$
\sum_{i\in I}|\alpha_{0,p}|\ll 1,
$$
and the $K_i$ are trace functions of pairwise geometrically
non-isomorphic sheaves $\sheaf{G}_i$ of weight $\leq 0$ modulo $p$,
with $\sheaf{G}_0$ trivial (see~\cite[Prop. 6.7]{FKM2}) and
$$
\cond(\sheaf{G}_i)\ll 1.
$$
\par
Using the method of moments, it follows easily that
Theorem~\ref{th-clt} follows from the following general result:

\begin{theorem}\label{th-clt2}
  With notation as above, let $K_p$ be trace functions modulo $p$
  which are geometrically irreducible and geometrically non-trivial,
  with conductor $\cond(K_p)\ll 1$.
\par
Let $\kappa$ and $\lambda\geq 0$ be integers. We have
$$
\lim_{p\ra +\infty} \frac{1}{p-1}\sum_{a\in\Fpt}E_f(X;p,a)^{\kappa}
\overline{E_f(X;p,a)}^{\lambda}K(a)=0.
$$
\end{theorem}

In turn, the method in~\cite[\S 3]{FGKM} and~\cite[\S 6.2, \S 7]{kr}
(based on the Voronoi summation formula) reduces this statement to the
following case of sums of products (where we again abbreviate
$\hypk_N(x)=\hypk_N(x;p)$):

\begin{theorem}
  Let $N\geq 2$ be an integer, and let $\kappa$, $\lambda\geq 0$ be
  integers, with $\lambda=0$ if $N$ is even. Let $p$ be a prime number
  and $K$ the trace function of a geometrically irreducible, not
  geometrically trivial, $\ell$-adic sheaf modulo $p$. We have
$$
\sum_{x\in\Fpt} \hypk_N(a_1x)\cdots \hypk_N(a_{\kappa}x)
\overline{\hypk_N(b_1x)\cdots \hypk_N(b_{\lambda}x)} K(x)\ll p^{1/2}
$$
with an implied constant depending only on $(\kappa,\lambda)$, for all
tuples $(a_i,b_j)$ in $(\Fpt)^{\kappa+\lambda}$ with at most
$$
C(\kappa,\lambda)p^{(\kappa+\lambda-1)/2}
$$
exceptions for some constant $C(\kappa,\lambda)\geq 0$ independent of
$p$.
\end{theorem}

Because of Examples~\ref{ex-old} (1) (for $N$ even) and~\ref{ex-old}
(2) (for $N$ odd), this statement follows immediately from
Theorems~\ref{th-control-1} and~\ref{th-control-2} in the next
section combined with Proposition~\ref{pr-rh}.

\section{A case of control of the diagonal}\label{sec-control}

The classification of diagonal cases of the previous section is
usually accompanied in applications by results dealing with these
diagonal cases.  Here is one typical instance, in the situation of
Example~\ref{ex-old}(2), which is the type of results used
in~\cite{FGKM} and~\cite{kr} (as explained in the previous section):

\begin{theorem}\label{th-control-1}
  Let $\sheaf{F}_0$ be a lisse $\ell$-adic sheaf on $\Gg_{m}$ over
  $\Fp$, which is pointwise pure of weight $0$ and self-dual with
  geometric monodromy group $G$ such that $G^0=\Sp_{r}$, and such that
$$
\Autz(\sheaf{F}_0)\cap \Tt=1,
$$
where $\Tt$ is the diagonal torus in $\PGL_2$.
\par
Fix a geometrically irreducible sheaf $\sheaf{G}$ lisse on a dense
open subset $U\subset \Gg_{m,\Fp}$ and a positive integer $k\geq
1$. The number of $k$-tuples $\uple{a}$ of elements of $\Fpt$ such
that
$$
H^2_c(U\times\bFp,\bigotimes_{a\in\uple{a}}[\times
a]^*\sheaf{F}_0\otimes\dual(\sheaf{G}))\not=0
$$
is bounded by $Cp^{k/2}$, where $C\geq 0$ is a constant depending only
on $k$. If $\sheaf{G}$ is geometrically non-trivial, the bound can be
improved to $Cp^{(k-1)/2}$.
\end{theorem}

\begin{proof}[Proof of Theorem~\ref{th-control-1}]
  We assume that there is at least one such $k$-tuple $\uple{a}$,
  since otherwise the bound is obvious. We then \emph{fix} such a
  tuple. 
\par
Then, let $\uple{a}^*$ denote the primitive tuple of distinct elements
of $\uple{a}$, and consider the tuple of sheaves
$\uple{\sheaf{F}}=([\times a]^*\sheaf{F}_0)_{a\in\uple{a}^*}$
restricted to $U$. By Example~\ref{ex-old}(2), it is
$U$-generous. Moreover, if $n_a\geq 1$ denotes the multiplicity of
$a\in\uple{a}^*$ in the tuple $\uple{a}$, we have
$$
\bigotimes_{a\in\uple{a}}[\times a]^*\sheaf{F}_0=
\bigotimes_{a\in\uple{a}^*}([\times a]^*\sheaf{F}_0)^{\otimes n_a}
=\sheaf{F}_{\uple{n}}
$$
with the notation of Theorem~\ref{th-diag}.
\par
By this theorem, the assumption that
$$
H^2_c(U\times\bFp,\bigotimes_{a\in\uple{a}}[\times
a]^*\sheaf{F}_0\otimes\dual(\sheaf{G}))\not=0
$$
therefore implies that there is a geometric isomorphism 
$$
\pi_{\uple{a}}^*\sheaf{G}\simeq \bigotimes_{a\in
  \uple{a}^*}\Lambda_a\Bigl(\pi_{\uple{a}}^*[\times a]^*\sheaf{F}_0\Bigr),
$$
of lisse sheaves, where $V\fleche{\pi_{\uple{a}}} U$ is a finite
abelian \'etale covering and where $\Lambda_a$ is some irreducible
representation of the group $G^0=\Sp_r$ such that $\Lambda_a$ is an
irreducible subrepresentation of the representation $\std^{\otimes
  n_a}$ of $G^0$.
\par
Now let $\uple{b}\not=\uple{a}$ be any $k$-tuple such that
$$
H^2_c(U\times\bFp,\bigotimes_{b\in\uple{b}}[\times b]^*\sheaf{F}_0
\otimes\dual(\sheaf{G}))\not=0,
$$
and let $\uple{b}^*$ denote the tuple of distinct elements of
$\uple{b}$. We then also have
$$
\pi_{\uple{b}}^*\sheaf{G}\simeq \bigotimes_{b\in
  \uple{b}^*}\tilde{\Lambda}_b\Bigl(\pi_{\uple{b}}^*[\times
b]^*\sheaf{F}_0\Bigr)
$$
for some representations $\tilde{\Lambda}_b$ of $G^0$ such that
$\tilde{\Lambda}_b$ is an irreducible subrepresentation of the
representation $\std^{\otimes n_b}$ of $G^0$.
% , where $n_b$ is the multiplicity
% of $b$ in $\uple{b}$.
By Lemma~\ref{lm-intersect} (after pulling back to the union of
$\uple{a}^*$ and $\uple{b}^*$), it follows that if we partition
$\uple{b}^*\sim (\uple{c},\uple{d})$ where $\uple{c}$ is the primitive
tuple of elements common to $\uple{a}$ and $\uple{b}$, and $\uple{d}$
is the rest, then we have
\begin{equation}\label{eq-condition}
%% \Lambda_a=\tilde{\Lambda}_a\text{ for } a\in\uple{c},\quad\quad
\tilde{\Lambda}_b=1\text{ for } b\in\uple{d}. 
\end{equation}
\par
We can partition any tuple $\uple{b}$ uniquely (up to order) as
$\uple{b}\sim (\uple{c}',\uple{d}')$ where $\uple{c}'$ has an
associated primitive tuple $\uple{c}$ which is a subtuple of
$\uple{a}^*$. We will count the number of possibilities for $\uple{b}$
to satisfy the non-vanishing condition by estimating the possibilities
for $\uple{c}'$ and $\uple{d}'$ separately.
% or in other words the possible multiplicities defining $\uple{c}'$
% in terms of $\uple{c}$ (and similarly for $\uple{d}$).
\par
% Note first that it is not possible that $\uple{c}$ be empty, since in
% that case $\sheaf{G}$ would be trivial. 
We first claim that the number of possible $\uple{c}'$ is bounded in
terms of $k$ only. Indeed, the number of possible primitive $\uple{c}$
is so bounded, simply because it is a subtuple of $\uple{a}^*$, and
for each fixed $\uple{c}$, the multiplicities allowed in $\uple{c}'$
for the components $c\in\uple{c}$ are at most $k$, so that the number
of $\uple{c}'$ is also bounded in terms of $k$ only.
\par
Now consider the potential $k$-tuples $\uple{b}=(\uple{c}',\uple{d}')$
where $\uple{c}$ is a fixed subtuple of
$\uple{a}^*$. From~(\ref{eq-condition}), the multiplicity $n_b\geq 1$
of any $b\in\uple{d}'$ is constrained by the condition that the
trivial representation is a subrepresentation of $\std^{\otimes
  n_b}$. In other words, since $G^0=\Sp_r$, the multiplicity must be
even, hence $\geq 2$. In particular, the size of the associated
primitive tuple $\uple{d}$ is at most $k/2$, and the number of
possibilities for $\uple{d}'$ is at most $p^{k/2}$ for any given
$\uple{c}'$.
\par
Combining these two bounds, we conclude, as claimed, that the number
of possible tuples $\uple{b}$ is $\leq C(k)p^{k/2}$. For the more
precise estimate when $\sheaf{G}$ is geometrically non-trivial, note
first that if the monodromy group $G$ is connected, then the tuple
$\uple{c}$ must be of size $\geq 1$ if $\sheaf{G}$ is geometrically
non-trivial, so that the bound for the size of $\uple{d}$
becomes $\leq (k-1)/2$ instead of $\leq k/2$. Thus only cases where
$G\not=G^0$ need be considered.
\par
Similarly, we are done unless $\uple{c}$ is empty, which means unless
$\pi_{\uple{a}}^*\sheaf{G}$ is trivial. This can only happen if the
rank of $\sheaf{G}$ is one. By the above, the tuples $\uple{b}$ that
may occur must have even multiplicity (in particular, $k$ is
even). The number of these where the associated primitive tuple has
size $<k/2$ is $\ll p^{(k-1)/2}$, so there only remains to estimate
the number of those of the form
\begin{equation}\label{eq-b-mirror}
\uple{b}=(b_1,b_1,b_2,b_2,\ldots,b_{k/2},b_{k/2})
\end{equation}
where the $b_i$ are distinct.  Then
$$
\bigotimes_{b\in\uple{b}}[\times b]^*\sheaf{F}_0\simeq 
\End\Bigl(\bigotimes_{i}[\times b_i]^*\sheaf{F}_0\Bigr).
$$
\par
By Lemma~\ref{lm-gkr}, the sheaf
$$
\bigotimes_{i}[\times b_i]^*\sheaf{F}_0
$$
is geometrically irreducible. In fact, if $G_{\uple{b}}$ denotes its
geometric monodromy group, the restriction of the corresponding
representation $\rho_{\uple{b}}$ to $G_{\uple{b}}^0$ is
irreducible. It follows that $\End(\rho_{\uple{b}})$ does not contain
any non-trivial one-dimensional character: indeed, each such character
is trivial on $G_{\uple{b}}^0$ (because the latter is semisimple), and
therefore the number of one-dimensional subrepresentations of
$\End(\rho_{\uple{b}})$ (with multiplicity) is at most equal to the
number of trivial subrepresentations of its restriction to
$G_{\uple{b}}^0$, which is equal to $1$ by Schur's Lemma. Since the
trivial representation occurs in $\End(\rho_{\uple{b}})$, there can be
no other character.
\par
This argument shows that, if $\sheaf{G}$ is a geometrically
non-trivial character, then no $\uple{b}$ of the
form~(\ref{eq-b-mirror}) with distinct $b_i$'s has the property that
$$
H^2_c(U\times\bFp,\bigotimes_{b\in\uple{b}}[\times
b]^*\sheaf{F}_0\otimes\dual(\sheaf{G}))\not=0
$$
and this concludes the proof.
\end{proof}

\begin{remark}
  (1) This bound is in general best possible, as the following example
  shows: take $k$ odd, and $\sheaf{G}=\sheaf{F}_0$ where $\sheaf{F}_0$
  has monodromy equal to $\Sp_r$. Then all
$$
\uple{a}=(1,a_2,a_2,\ldots, a_{(k-1)/2},a_{(k-1)/2})
$$
with $(a_2,\ldots, a_{(k-1)/2})$ taken in $U(\Fp)$ satisfy the desired
non-vanishing. The number of such tuples is $\sim p^{(k-1)/2}$ for $p$
large (provided $\Pp^1-U$ has bounded size).  However, as we will see,
the $k$-tuples that arise can be classified to some extent, and in
many cases, better bounds can be obtained.
\par
(2) The result contrasts strongly with some cases where a tuple of sheaves
is constructed from a sheaf $\sheaf{F}_0$ in such a way that it is not
generous: for instance, take $\sheaf{F}_0=\sheaf{L}_{\psi(X^{-1})}$ on
$\Gg_m$; then for any $k$-tuple $\uple{a}$, we have
$$
\bigotimes_{i}[\times a_i]^*\sheaf{F}_0\simeq
\sheaf{L}_{\psi(f_{\uple{a}}(X))},
$$
where
$$
f_{\uple{a}}(X)=\Bigl(\sum_{i}{\frac{1}{a_i}}\Bigr)\frac{1}{X},
$$
and if we take simply $\sheaf{G}=1$, we find that all $k$-tuples with
$$
\frac{1}{a_1}+\cdots+\frac{1}{a_k}=0
$$
satisfy
$$
H^2_c(\Gg_m\times\bFp,\bigotimes_{i}[\times a_i]^*\sheaf{F}_0)\not=0.
$$
\par
Obviously, the number of these tuples is about $p^{k-1}$, which is
larger (for $k\geq 3$) than in the generous case.
\end{remark}

Another case which is proved in a similar manner is:

\begin{theorem}\label{th-control-2}
  Let $\sheaf{F}_0$ be a lisse $\ell$-adic sheaf on $\Gg_{m}$ over
  $\Fp$, which is pointwise pure of weight $0$ and with geometric
  monodromy group $G$ such that $G^0=\SL_{r}$ with $r\geq 3$, and such
  that the projective automorphism group of $\sheaf{F}_0$ is trivial.
\par
Fix a geometrically irreducible sheaf $\sheaf{G}$ lisse on a dense
open subset $U\subset \Gg_{m,\Fp}$ and positive integers $k\geq 0$ and
$l\geq 0$ with $k+l\geq 1$. The number of pairs $(\uple{a},\uple{b})$
of $k$-tuples $\uple{a}$ and $l$-tuples $\uple{b}$ of elements of
$\Fpt$ such that
$$
H^2_c(U\times\bFp,\bigotimes_{a\in\uple{a}}[\times
a]^*\sheaf{F}_0\otimes\bigotimes_{b\in\uple{b}}[\times
b]^*\dual(\sheaf{F}_0) \otimes\dual(\sheaf{G}))\not=0
$$
is bounded by $C(k,l)p^{(k+l)/2}$, where $C(k,l)\geq 0$ is a constant
depending only on $k$ and $l$ only. If $\sheaf{G}$ is geometrically
non-trivial, the bound can be improved to $C(k,l)p^{(k+l-1)/2}$.
\end{theorem}

In the proof, the main difference with the previous case is that the
condition that the trivial representation be a subrepresentation of
$\std^{\otimes n}\otimes \dual(\std^{\otimes m})$ of $\SL_r$ is that
$r\mid n-m$, as recalled in the proof of Theorem~\ref{th-main1}.

\section{How to use the results}\label{sec-howto}

We explain here quite informally how an analytic number theorist might
go about using the results of this paper concretely.  In particular,
we will attribute to trace functions $K$ some properties which
properly are only defined for sheaves (e.g., irreducibility).
\par
We assume that a concrete problem gives rise to a sum
$$
\sum_{x\in\Fp} K_1(x)^{\sigma_1}K_2(x)^{\sigma_2}\cdots
K_k(x)^{\sigma_k}\overline{M(x)}
$$
where the $K_i$ and $M$ are some functions defined on $\Fp$ and
$K_i(x)^{\sigma_i}$ is either $K_i(x)$ or $\overline{K_i(x)}$. The
question is to estimate this sum, and the main variable should be $p$,
which will tend to infinity.
\par
To handle this sum, one should first check whether it is of the type
described in the introduction, that is, whether
$K_i(x)=K(\gamma_i\cdot x)$ for some elements of $\PGL_2(\Fp)$ and
some fixed function $K$.  If this is the case, we suggest steps in the
next subsection, and otherwise in the following one.
\par
This ``howto'' may lead to a proof that the sum under investigation
has square-root cancellation; it may also simply suggest whether this
is the case or not, leaving some algebraic confirmations for a
rigorous proof. In any case, it should help clarify the situation.

\subsection{Sums of products with fractional linear transformations}

We assume here that $K_i(x)=K(\gamma_i\cdot x)$.  The following steps
may then help, where any negative answer to the questions means that
one should look at the more general case of the next subsection:
\begin{enumerate}
\item Is the function $K$ a trace function of weight $0$ over $\Fp$,
  and is $M(x)=e(hx/p)$ for some $h\in\Fp$?  To answer this, one can
  very often just refer to lists of examples of trace functions, and
  to their formal stability properties to construct new ones from
  known trace functions; the weight $0$ condition can often be
  obtained by normalization.
\item Assuming a positive answer to the previous question, one should
  then estimate the conductors of $K$ and $M$; this is often an easy
  matter, and the most relevant issue is that the conductor should be
  bounded independently of $p$ in order to get a good estimate from
  the Riemann Hypothesis.
\item What is the geometric monodromy group $G$ of $K$? This will
  usually be the most delicate part, and one should rely mostly on the
  examples accumulated in the many works of Katz (for
  instance~\cite{katz-gkm,katz-esde,katz-sarnak}). If $G$ is neither
  $\SL_r$ nor $\Sp_r$, one should go to the general setting of the
  next subsection.
\item Assuming that $G$ is either $\SL_r$ or $\Sp_r$, what is the
  projective automorphism group $\Gamma$ of $K$ (defined
  in~(\ref{eq-autz}))? Concretely, even if this is not entirely
  equivalent, what are the elements $\gamma\in\PGL_2(\Fp)$ such that
$$
K(\gamma\cdot x)=\lambda(x) K(x)
$$
for all $x\in \Fp$?  Is $\Gamma$ trivial?  Although this computation
is usually much easier than that of $G$, it may not be easy to find an
answer in the literature because this group has not been computed as
systematically as the geometric monodromy group.
\item Assuming $\Gamma$ is trivial, and $G$ is $\SL_r$, does $K$ have
  a special involution, i.e., roughly speaking, does there exist an
  involution $\xi$ such that 
$$
K(\xi\cdot x)=\lambda(x)\overline{K(x)}
$$
with $|\lambda(x)|=1$ for all $x$?  (For instance, $\xi\cdot x=1/x$ or
$\xi\cdot x=-x$ are the most common).
\item If one knows the answer to these questions, then
  Corollary~\ref{cor-concrete} gives (almost) a characterization of
  when the sum has square-root cancellation, uniformly in $p$, since
  $K$ is then the trace function of a bountiful sheaf (up to maybe
  tweaking $K$ at a bounded number of points to reduce to a
  middle-extension sheaf).
\end{enumerate}

\subsection{General sums of products}

We assume here that the sum to handle is not of the type
$K_i(x)=K(\gamma_i\cdot x)$ with $M(x)=e(hx/p)$.  The following may
then help to apply our general results:
\begin{enumerate}
\item Are the functions $K_i$ trace functions over $\Fp$?  To answer
  this, one can very often just refer to lists of examples of trace
  functions, and to their formal stability properties to construct new
  ones from known trace functions.
\item Is $M$ a trace function?  If yes is it geometrically
  irreducible?  If the answer is ``no'', can one decompose $M$ as a
  combination of geometrically irreducible trace functions (as
  in~(\ref{eq-ft}) or~(\ref{eq-artin})) $M_j$?  If yes, then the sums
  with each $M_j$ should be studied;
\item Assuming $K_i$ and $M$ are trace functions, $M$ geometrically
  irreducible, one should then estimate the conductors of these trace
  functions; this is often an easy matter, and the most relevant issue
  is that the conductor should be bounded independently of $p$ in
  order to get a good estimate from the Riemann Hypothesis.
\item What are the geometric monodromy groups of the $K_i$, and their
  connected component of the identity? Are they ``big''? As already
  indicated, this is often delicate, because on the one hand rather
  precise information is needed, and on the other hand, determining
  this group in a ``new'' case is most often rather deep and difficult
  to handle by hand if one does not find the result in the works of
  Katz. If one knows the geometric monodromy groups, then one should
  check whether the connected component of the identity belongs to the
  list of groups in Section~\ref{ssec-general}.  If not (especially
  for $\SO_{2r}$), then some new argument is probably needed.
\item Assuming all geometric monodromy groups fit the list, do there
  exist $i\not=j$ such that~(\ref{eq-gkr-cond}) holds? In practice,
  this means, does there exist $i\not=j$ such that
\begin{equation}\label{eq-repeat}
K_i(x)=\lambda(x)K_j(x),\text{ or } \overline{K_i(x)}=\lambda(x)K_j(x)
\end{equation}
for all $x$, where $|\lambda(x)|=1$? This might be a delicate matter
to settle, but usually such identities are either obvious or do not
exist (it is also often possible to investigate this possibility
experimentally).
\item If one finds such a pair, say $(i_0,j_0)$, then one should
  replace $K_{i_0}(x)$ by $K_{j_0}(x)$ or $\overline{K_{j_0}(x)}$ and
  increase the multiplicity of $K_{j_0}$ or its dual; then one repeats
  the last two steps until the sum is expressed as
$$
\sum_{x} \prod_{i\in I}K_i(x)^{m_i}\overline{K_i(x)}^{\, n_i}\overline{M(x)}
$$
where $m_i+n_i\geq 1$ and, among the $K_i$ for $i\in I$, no
``repetition'' as in~(\ref{eq-repeat}) occurs.
\item At this point, the result of Section~\ref{sec-general} apply to
  the family $(K_i)_{i\in I}$; thus Theorems~\ref{th-diag} (when all
  $n_i=0$) or~\ref{th-diag-2} are applicable, and give a sufficient
  condition for square-root cancellation, in terms of $M$.  This
  criterion may be difficult to exploit, but if all geometric
  monodromy groups are connected, it means that $M$ splits as a
  product
$$
M(x)=\prod_{i\in I}M_i(x)
$$
such that \emph{all} the sums
$$
\sum_x K_i(x)^{m_i}\overline{K_i(x)}^{\, n_i}\overline{M_i(x)}
$$
are large. This might again be somewhat delicate to exclude without
algebraic tools, but should help get an intuitive understanding of
what is true about the original sum.
\end{enumerate}

\end{document}